\documentclass{amsart}

\usepackage[T1]{fontenc}
\usepackage[utf8]{inputenc}
\usepackage[USenglish]{babel}
\usepackage{overpic}

\usepackage[normalem]{ulem} 

\usepackage[
	bookmarks=true,
	plainpages=false,
	linktocpage,
	colorlinks=true,
	citecolor=green!80!black,
	linkcolor=red!70!black,
	filecolor=magenta,
	urlcolor=magenta,
	breaklinks,
	pdfauthor={Martin Winter},
]{hyperref}

\usepackage{amsmath,amsthm}
\usepackage{mathtools} 

\iftrue
\usepackage{amssymb} 
\else
\usepackage[charter,cal=cmcal]{mathdesign}
\fi

\usepackage{wasysym} 

\usepackage{calc} 

\usepackage{colortbl,color} 
\usepackage[dvipsnames]{xcolor}
\colorlet{midgray}{lightgray!50!gray}

\usepackage{centernot} 
\usepackage{array} 
\usepackage{enumitem,moreenum} 

\usepackage[noadjust]{cite} 

\usepackage[nameinlink,capitalize,noabbrev]{cleveref}
\usepackage{nicefrac}

\usepackage{tikz-cd} 

\usepackage[font=small,labelfont=bf]{caption}

\usepackage{blkarray}

\usepackage[
    colorinlistoftodos,
    backgroundcolor=yellow,
    textsize=footnotesize
]{todonotes}

\newcommand{\note}[1]{\todo[color=lightgray]{#1}}

\newcommand{\RR}{\mathbb{R}}    

\newcommand{\T}{\mathcal{T}}
                   
\newcommand{\Sph}{\mathbb{S}}   

\def\^#1{^{(#1)}}
\def\s^#1{^{\smash{(#1)}}}

\def\:{\colon}

\newcommand{\cupdot}{\mathbin{\mathaccent\cdot\cup}}

\newcommand{\labelstyle}[1]{\upshape(\textit{#1})}
\newcommand{\mylabel}{\labelstyle{\roman*}}
\newenvironment{myenumerate}
    {\begin{enumerate}[label=\mylabel]}
    {\end{enumerate}}

\def\itm#1{\labelstyle{\romannumeral#1\relax}}
\def\hitm#1{\hyperref[itm:#1]{\itm{#1}}}
\def\itmto#1#2{{\itm#1 $\Longrightarrow$ \itm#2}}

\makeatletter
\newcommand{\myitem}[1]{%
\item[#1]\protected@edef\@currentlabel{#1}%
}
\makeatother

\newcommand{\freespace}{\kern.07em} 
\newcommand{\free}{\freespace\cdot\freespace} 

\newcommand{\TODO}{{\footnotesize\textcolor{red}{TODO}}}

\colorlet{xxxcolor}{lightgray}
\colorlet{mcolor}{red}
\colorlet{acolor}{blue}
\colorlet{tcolor}{green}

\newcommand{\msays}[1]{{\footnotesize\textcolor{mcolor}{\textbf M: #1}}}

\newenvironment{mnote}
{
    \begingroup
    \footnotesize\color{mcolor} \textbf M:%
}
{ 
    \endgroup
}

\theoremstyle{plain}  
\newtheorem{theorem}{Theorem}[section]
\newtheorem{corollary}[theorem]{Corollary}
\newtheorem{lemma}[theorem]{Lemma}
\newtheorem{proposition}[theorem]{Proposition}
\newtheorem{conjecture}[theorem]{Conjecture}
\newtheorem{problem}[theorem]{Problem}
\newtheorem{question}[theorem]{Question}

\newtheorem*{theorem*}{Theorem}

\theoremstyle{definition} 
\newtheorem{definition}[theorem]{Definition}
\newtheorem{example}[theorem]{Example}
\newtheorem{remark}[theorem]{Remark}
\newtheorem{observation}[theorem]{Observation}

\crefname{theorem}{Theorem}{Theorems}
\crefname{proposition}{Proposition}{Propositions}
\crefname{lemma}{Lemma}{Lemmas}
\crefname{corollary}{Corollary}{Corollaries}
\crefname{remark}{Remark}{Remarks}
\crefname{example}{Example}{Examples}
\crefname{definition}{Definition}{Definitions}
\crefname{problem}{Problem}{Problems}
\crefname{observation}{Observation}{Observation}
\crefname{construction}{Construction}{Construction}

\theoremstyle{theorem}
\newtheorem{innercustomgeneric}{\customgenericname}
\providecommand{\customgenericname}{}
\newcommand{\newcustomtheorem}[2]{%
  \newenvironment{#1}[1]
  {%
   \renewcommand\customgenericname{#2}%
   \renewcommand\theinnercustomgeneric{##1}%
   \innercustomgeneric
  }
  {\endinnercustomgeneric}
}

\newcustomtheorem{theoremX}{Theorem}
\newcustomtheorem{lemmaX}{Lemma}
\newcustomtheorem{conjectureX}{Conjecture}

\DeclareMathOperator{\id}{id}

\DeclareMathOperator{\dist}{dist}

\DeclareMathOperator{\Int}{int}  	
  	
\DeclareMathOperator{\cl}{cl}

\let\eset=\varnothing

\let\<=\langle
\let\>=\rangle

\def\...{...}
\newcommand{\shortStyle}{\textit}
\newcommand{\ie}{\shortStyle{i.e.,}}

\newcommand{\eg}{\shortStyle{e.g.}}

\newcommand{\wrt}{\shortStyle{w.r.t.}}

\newcommand{\cf}{\shortStyle{cf.}}

\newcommand{\resp}{resp.}

\def\nlspace{\nolinebreak\space}
\def\nspace{\nlspace}
\def\nls{\nlspace}

\makeatletter
\renewcommand*{\eqref}[1]{%
  \hyperref[{#1}]{\textup{\tagform@{\ref*{#1}}}}%
}
\makeatother

\numberwithin{equation}{section}

\makeatletter
\newcommand{\addresseshere}{%
  \enddoc@text\let\enddoc@text\relax
}
\makeatother

\usepackage{inconsolata}

\usepackage{listings}

\lstnewenvironment{cpp}{\minipage[c]{0.95\linewidth}}{\endminipage}
\lstnewenvironment{cpp*}[2][]
	{\minipage[c]{0.95\linewidth}
	 \lstset{morekeywords=[3]{#1},morekeywords=[2]{#2}}}
	{\lstset{deletekeywords=[2]{#1},deletekeywords=[3]{#2}}
	 \endminipage}
	
\newcommand{\tempnewpage}{}

\newcommand{\Tam}{T. T$\hat{\mathrm{a}}$m Nguy$\tilde{\hat{\mathrm{e}}}$n-Phan}

\newcommand{\cpx}{X}
\newcommand{\FKT}{\ensuremath{X_{\mathrm{FKT}}}}

\newcommand{\skel}[1]{\ensuremath{G_{#1}}}

\newcommand{\PLsets}{\PL-sets}

\newcommand{\Creg}{\ensuremath{X_{\mathrm{reg}}}}

\newcommand{\Xany}{X}
\newcommand{\Xfull}{\ensuremath{\Xany_{\mathrm{full}}}}
\newcommand{\Xreg}{\ensuremath{\Xany_{\mathrm{reg}}}}
\newcommand{\Xind}{\ensuremath{\Xany_{\mathrm{ind}}}}

\newcommand{\embd}[1]{\ensuremath{{#1}^\phi}}

\newcommand{\Y}{\ensuremath{\mathrm{Y}}}
\newcommand{\PL}{\ensuremath{\mathrm{PL}}}
\newcommand{\DY}{\ensuremath{\Delta\kern-1.2pt \Y}}
\newcommand{\YD}{\ensuremath{\Y\kern-1.2pt\Delta}}
\newcommand{\trafo}{trans\-for\-ma\-tion}
\newcommand{\trafos}{trans\-for\-ma\-tions}
\newcommand{\DYtrafo}{\DY-\trafo}
\newcommand{\DYtrafos}{\DY-\trafos}
\newcommand{\YDtrafo}{\YD-\trafo}
\newcommand{\YDtrafos}{\YD-\trafos}

\newcommand{\vdH}{van~der~Holst}
\newcommand{\VdH}{Van~der~Holst}
\newcommand{\cdV}{Colin de Verdi\`ere}

\newcommand{\defi}[1]{{\color{darkgray}\emph{#1}}}
\newcommand{\exm}[1]{\mathrm{Ex}(#1)}
\newcommand{\comment}[1]{}
\newcommand{\RRf}{\ensuremath{\RR^4}}

\renewcommand{\medskip}{}

\usepackage{scalerel}[2016/12/29]

\def\circsymb{\circ}
\def\disksymb{\bullet}
\def\starsymb{\star}
\def\wheelsymb{{\star\kern1pt\mathllap\ocircle}}
\def\Xcirc{X_\circsymb}
\def\Xdisk{X_\disksymb}
\def\Xstar{X_\starsymb}

\def\Gcirc{G_\circsymb}

\def\Gstar{G_\starsymb}

\begin{document}


\title
[On 2-complexes embeddable in 4-space and their underlying graphs]
{On 2-complexes embeddable in 4-space, and the excluded minors of their underlying graphs}

\author[A. Georgakopoulos]{Agelos Georgakopoulos{${}^\dagger$}}
\thanks{{${}^\dagger$}Supported by EPSRC grants EP/V009044/1, EP/V048821/1, and EP/Y004302/1.}
\author[M. Winter]{Martin Winter{${}^\ddagger$}}
\thanks{{${}^\ddagger$}Supported by EPSRC grant EP/V009044/1.}

\thanks{Mathematics Institute, University of Warwick, Coventry CV4 7AL, United Kingdom}

		




\maketitle


\date{\today}

\begin{abstract}

    We study the potentially undecidable problem of whether a given 2-dimensional CW complex can be embedded into $\mathbb{R}^4$.
    We provide operations that preserve embeddability, including joining and cloning of 2-cells, as well as \DYtrafos.
    We also construct a CW complex for which \YDtrafos\ do not preserve embeddability. 

    We use these results to study 4-flat graphs, \emph{i.e.}, graphs that embed in $\mathbb{R}^4$ after attaching any number of 2-cells to their cycles; a graph class that naturally generalizes planarity and linklessness. We verify several conjectures of van der Holst; in particular, we prove that each of the 78 graphs of the Heawood family is an excluded minor for the family of 4-flat graphs.    
\end{abstract}


\section{Introduction}
\label{sec:introduction}

\iftrue 

The study of planar graphs marks the starting point of both topological graph theory and graph minor theory.
%
A variety of concepts have been introduced with the goal of capturing a 
higher-dimension analogue of planarity.
Linkless graphs, for example, are defined by the existence of particular embeddings into 3-space, and can likewise be characterized by a short list of excluded minors~\cite{RoSeThoSac}. 
In 2006 van~der~Holst introduced a natural 4-dimensional analogue \cite{van2006graphs}: a graph $G$ is \mbox{\emph{4-flat}} if every 2-dimensional regular CW complex having $G$ as its 1-skeleton can be (piecewise linearly) embedded into $\RR^4$.
%
%
Van der Holst's work contains a number of results and plausible conjectures that paint the picture of 4-flat graphs as the most natural continuation from planar and linkless to 4-space. 
In particular, 4-flat graphs encompass planar and linkless graphs and appear to have the right behaviour in terms of forbidden minors and the Colin de Verdiére graph invariant. 
They moreover allow natural generalisations to even higher dimensions. 

The study of 4-flat graphs is intimately connected with the embedding problem $2\to 4$, \ie\ 
deciding the embeddability of  a 2-dimensional CW complex into $\RR^4$. 
A fundamental homological obstruction to embeddability $n\to 2n$, the van Kampen obstruction, 
fails to be sufficient precisely for $n=2$ \cite{freedman1994van}, and it is a prominent open question whether the decision problem $2\to 4$ is  decidable \cite{matouvsek2010hardness}.
As a consequence, seemingly simple questions about 4-flat graphs can be hard, since there is currently no way to decide whether a given graph is 4-flat. 

\medskip
In this paper we provide operations on 2-dimensional CW complexes that preserve embeddability into $\RR^4$, and use them to make a number of conclusions about 4-flat graphs.\nls We answer several conjectures of van der Holst regarding their excluded minors and on operations that preserve 4-flatness.
Moreover, our tools allow us to give streamlined proofs for some known results.


\subsection{Background on 4-flat graphs}

We now recall some results of van der Holst  \cite{van2006graphs}. 
The 4-flat graphs form a minor-closed graph family and are therefore, by~the Robertson-Seymour Graph Minor theorem \cite{GMXX}, characterized by a finite list~of \emph{excluded minors}, \ie\ graphs that are not 4-flat, but every proper minor of them~is. The complete list of excluded minors is unknown, though there is a plausible candidate: 
the \emph{Heawood family} consists of the 78 graphs built from $K_7$ and $K_{3,3,1,1}$ by repeated application of \DY- and \YDtrafos.
The members of the~Heawood family are collectively known as ``Heawood graphs'', owing to the fact that  \emph{the} Heawood graph (\cref{fig:Heawood_Graph}, right) is  a member of this family. 
Van der Holst showed that the Heawood graphs are not 4-flat 
and proposed:

\begin{figure}
    \centering
    \includegraphics[width=0.73\textwidth]{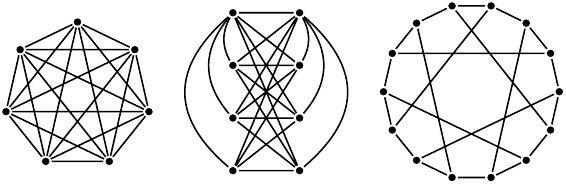}
    \caption{Three graphs from the Heawood family: $K_7$, $K_{3,3,1,1}$ and~\emph{the} Heawood graph. All three are not 4-flat.}
    \label{fig:Heawood_Graph}
\end{figure}

\begin{conjecture}[{\cite{van2006graphs}}]
    \label{vdH_Heawood}
    \label{conj:Heawood}
    Each of  the 78 graphs of the Heawood family is an excluded minor for the class of 4-flat graphs. Even more, there are no further excluded~minors.\footnote{The first statement of \cref{conj:Heawood} can be found at the end of \cite[\S 3]{van2006graphs}. The second statement is more implicit: it is conjectured that the graphs in the Heawood family are all excluded minors for the class of graphs with $\mu\leq 5$ \cite[\S 4]{van2006graphs}, and that a  graph has $\mu\leq 5$ if and only if it is 4-flat (\cite[Conjecture 3]{van2006graphs}). Combining the latter two conjectures yields the second statement.}
\end{conjecture}

This conjecture is natural, given the 
similarity 
of the Heawood family to 
the well-known Petersen family (that is, the family of excluded minors for linkless graphs): both are obtained from the Kuratowski graphs $K_5$ and $K_{3,3}$ by forming suspensions, and then \DY- and \YDtrafos.  

%
%
%
%
%

Van der Holst proved that both $K_7$ and $K_{3,3,1,1}$ are excluded minors, though~left this open for the other graphs of the Heawood family. We shall fill this gap with \cref{res:all_Heawood_are_excluded}, proved in~\cref{sec:Heawood_excluded_minor}.

Van der Holst proposed a different approach for proving this, based on two further conjectures:

\begin{conjecture}[{\cite[Conjecture 1]{van2006graphs}}]
    \label{vdH_doubling_edges}
    \label{conj:doubling_edges}
    \label{conj:cloning_edges}
    Cloning an edge of a graph (\ie\ replacing~it by two parallel edges) preserves 4-flatness.
\end{conjecture}

\begin{conjecture}[{\cite[Conjecture 2]{van2006graphs}}]
    \label{vdH_DY_YD}
    \label{conj:DY_YD}
    \DY- and \YDtrafos\ of graphs preserve 4-flat\-ness (\cf\ \cref{fig:DY_YD}).
\end{conjecture}

\begin{figure}[h!]
    \centering
    \includegraphics[width=0.51\textwidth]{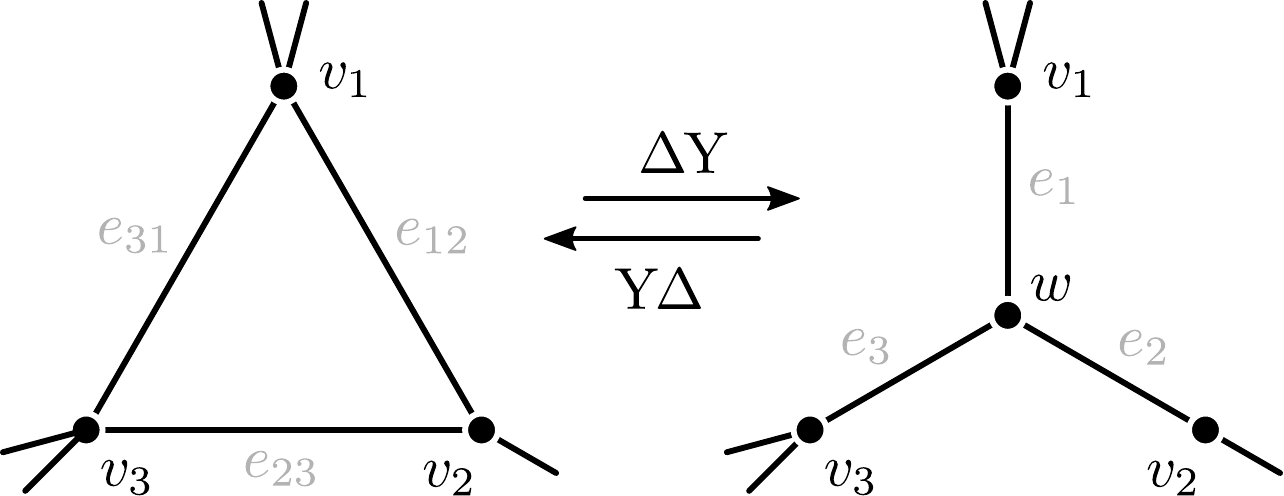}
    \caption{Visualization of \DY- and \YDtrafos.}
    \label{fig:DY_YD}
\end{figure}

We will confirm \cref{conj:cloning_edges} as well as the \DY-part of \cref{conj:DY_YD}. We also provide constructions for complexes and graphs that highlight the difficulties in proving the \YD-part of the conjecture.

\medskip
Another striking analogy to planar and linkless graphs might exist in connection to the Colin de Verdiére graph invariant $\mu(G)$ \cite{van1999colin}: recall that a graph $G$ is planar if and only $\mu(G)\le3$, and is linkless if and only if $\mu(G)\le 4$.
Van der~Holst~conjectured:

\begin{conjecture}[{\cite[Conjecture 3]{van2006graphs}}]
    \label{conj:Colin_de_Verdiere}
    A graph $G$ is 4-flat if and only if it has Colin de Verdiére invariant $\mu(G)\le 5$.
\end{conjecture}


%



We list some evidence in favor of this conjecture: all Heawood graphs have~$\mu=6$.
It moreover holds that $G$ is linkless if and only if the \emph{suspension} $G*K_1$, obtained by adding a new dominating vertex, is a 4-flat graph. One direction of this was proven in \cite[Theorem 2]{van2006graphs}, and we prove the other direction in \cref{res:linkless_planar_outerplanar}.
It follows that $G*K_1$ is 4-flat if and only if $\mu(G*K_1)\le 5$.


\medskip
An alternative definition of 4-flat graphs based on ``almost embeddings'' as well as generalizations to even higher dimensions are introduced in \cite{van2009graph}. 
Based on this the authors define a graph invariant $\sigma(G)$ and show that it agrees with $\mu(G)$ on linkless graphs, but disagrees at sufficiently large values. 
It is unknown whether $\sigma(G)\le 5$ is equivalent to being 4-flat. It is known that $\sigma(G)\le \mu(G)$ \cite{kaluvza2022even}.

\subsection{Overview and results}

The results of this paper can be divided into results on general 2-dimensional CW complexes and result on 4-flat graphs. 

\cref{sec:operations} explores operations on general 2-dimensional CW complexes, and especially on when those preserve embeddability into $\RR^4$.

\begin{theorem*}
    The class of 2-complexes embeddable into $\RR^4$ is closed under each of the following operations:
    \begin{myenumerate}
        \item joining 2-cells at vertices and edges (\cref{res:joining_embeds}).
        \item cloning 2-cells (\cref{res:cloning_embeds}).
        \item collapsing 2-cells (\cref{res:collapsing}).
        \item cloning edges (\cref{res:cloning_edges}).
    \end{myenumerate}
\end{theorem*}

We also obtain further results on rerouting 2-cells (\cref{res:rerouting_embeds}), stellifying \mbox{cycles} (\cref{res:stellifying}) (including the important special case of \DYtrafos; \cref{res:DY_YD}), and merging parallel edges (\cref{res:merging_iff_filling}).
These operations preserve embeddability only under additional assumptions and so the precise statements~require more care.

As an application of these operations, we prove the following:
\begin{theoremX}{\ref{res:all_2_cells_vw}}
Suppose a 2-dimensional CW complex $X$ has two vertices $v,w$ such that each 2-cell $c\subseteq \cpx$~is~incident to at least one of $v,w$. Then $\cpx$ can be embedded in $\RR^4$.
\end{theoremX}

To appreciate the strength of \cref{thm ops}, we remark that it is non-trivial to prove that $X$ is 4-embeddable even if its 1-skeleton has only a single vertex; see \cite{MathOverflow19618}, where this statement is deduced from a difficult theorem of Stallings. 
Moreover, \cref{thm ops} becomes false if we replace $v,w$ by a triple of vertices; a counterexample is the triple cone over any non-planar graph, as was proven by Grünbaum \cite{grunbaum1969imbeddings}.

\medskip
The above tools will allow us to verify van der Holst's \cref{conj:cloning_edges} (\cref{res:cloning_edges_4_flat}) and the \DY-part of \cref{conj:DY_YD} (\cref{res:DY_YD_4_flat}) on 4-flat graphs.

\cref{sec:YD} explores the intricacies of the inverse operations -- reverse stellification, and in particular, \YDtrafos. 
We show that the \YD-part of \cref{conj:DY_YD} is at the boundary of what can be true: stellification at cycles~of~length~$\ell\ge 4$ does not preserve 4-flatness (\cref{ex:reverse_stellifying_4_cycle_4_flat}); and we construct an embeddable~complex (based on the Freedman-Krushkal-Teichner complex \cite{freedman1994van}) having a \YDtrafo\ which is not embeddable (\cref{sec:YD_counterexample}).

From \cref{sec:4_flat} on we explore the implications of our results to the theory of~4-flat graphs and their full complexes.
Combined with  results from \cref{sec:operations}, the following operations are shown to preserve 4-flatness:

\begin{theorem*}
    The class of 4-flat graphs is closed under each of the following operations:
    \begin{myenumerate}
        \item cloning edges (\cref{res:cloning_edges_4_flat}).
        \item \DYtrafos (\cref{res:DY_YD_4_flat}).
        \item $k$-clique sums for 
        $k\leq 3$ (\cref{res:3_clique_sums}).
    \end{myenumerate}
\end{theorem*}

Other results obtained in \cref{sec:4_flat} include:
\begin{itemize}
    \item we consider alternative notions of 4-flatness: instead of \emph{every regular} complex on $G$ being embeddable, we require that \emph{every} (not necessarily regular) complex, or \emph{every induced} complex (\ie\ 2-cells only along induced cycles) is embeddable.
    We show that, surprisingly, both modifications give rise to the same notion of 4-flatness, despite being seemingly stronger and weaker respectively (\cref{res:variants_equal}).
    \item we show that 4-flat graphs are ``locally linkless'' (\cref{res:locally_linkless}).
    \item we show that a suspension $G*K_1$ 
    yields a 4-flat graph if and only if $G$ is linkless (\cref{res:linkless_planar_outerplanar}).
    \item we show that there is no simple inclusion relation between 4-flat and knotless graphs 
     (\cref{sec:4flat_knotless}).
\end{itemize}

Finally, in \cref{sec:Heawood_excluded_minor}, we verify the first statement of \cref{conj:Heawood}:

\begin{theoremX}{\ref{res:all_Heawood_are_excluded}}
\label{res:all_Heawood_are_excluded}
    All graphs of the Heawood family are excluded minors for the class of 4-flat graphs.
\end{theoremX}

We emphasize that this in not a priori a finite problem: recall that there is no known algorithm for deciding whether a given 2-complex embeds into $\RRf$, and it might be that none exists. There is also 
no known algorithm for deciding whether a given graph is 4-flat. 

The Graph Minor Theorem \cite{GMXX} 
says that every minor-closed class $F$ of finite graphs is characterised by a finite set 
$\exm{F}$ of excluded minors. However, the proof is inherently non-constructive, and it is often hopelessly difficult to determine $\exm{F}$ given $F$ even if the latter has a natural definition. Indeed, Fellows \& Langston \cite[Theorem~8]{FelLanSea} prove that there is no algorithm to compute the set of excluded minors of $F$ even if 
we are given a Turing machine that accepts precisely the graphs in $F$.

This theoretical difficulty is hard to overcome even for  specific classes $F$ of interest: there are more than 16 thousand excluded minors for embeddability into the torus \cite{GaMyChFor}, 
and more than 68 billion excluded minors for \Y$\Delta$\Y-reducible graphs  \cite{YuMor}.  On the other hand, some natural classes $F$ have a small $\exm{F}$ that we know exactly; these include outerplanar, series-parallel, planar, and linkless graphs \cite{RoSeThoSac}.  
Thus naturally defined minor-closed families $F$ for which we can find $\exm{F}$ exactly are rare. 
The best next candidate for such a class might be the class of 4-flat graphs, due to \cref{vdH_Heawood}.

\fi 

\tempnewpage

\section{Basic terminology and examples}
\label{sec:notation}

\subsection{Graphs, complexes and embeddings}

Throughout this paper $G$ denotes a finite graph or \emph{multi-graph}, i.e.\ we allow loops and parallel edges.
We consider graphs as 1-dimensio\-nal CW complexes, \ie\ as topological spaces with 0-cells (the vertices) and 1-cells (the edges).

We write $X$ to denote a \emph{2-dimensional CW complex} (or just \emph{2-complex} for short). 
A 2-complex over $G$ is obtained by attaching homeomorphs of the disk $\mathbb D^2$ ---called  \emph{2-cells}--- along closed walks in $G$.
Each 2-cell $c\subset X$ is determined by its \emph{attachment map} $\partial c\:\partial \mathbb D^2\to G$.
%
If a 2-cell $c$ is attached along a cycle, \ie\ a closed walk without self-intersections, and the attachment map is injective, then  $c$ is said to be \emph{regular}. We say that $X$ is \emph{regular}, if each of its 2-cells is.
Given a complex $X$, the underlying graph, called the \emph{1-skeleton} of $X$, is denoted by $G_X$. 
%
Given a cell $c\subseteq X$, we write $X-c$ for the subcomplex with this cell removed (but its boundary $\partial c$ left intact). 


All subsets of $\RR^d$, all embeddings into $\RR^d$ and all homeomorphisms between~such sets, are assumed to be piecewise linear (PL). We follow the terminology of  \cite{rourke2012introduction} for piecewise linear topology.
%
%
Recall that if $Y$ is a \PL-set then for
each point $y\in Y$ and for each sufficiently small ball (or star-shaped neighborhood) $B\subset\RR^d$,\nls the~in\-tersection $B\cap Y$ coincides with the cone over $\partial B\cap Y$ with apex at $y$.
Each neighborhood $B$ with this property is called a \emph{link neighborhood} of $y$, and $\partial B\cap Y$ is a \emph{link} of $Y$ at $y$.


All embeddings  in this paper are, if not stated otherwise, into $\RR^4$.
Given an~embedding $\phi\:X\to\RR^4$, we write $X^\phi$ instead of $\phi(X)$ to denote the image of~$\phi$.
A complex $X$ is \emph{embeddable} (or \emph{4-embeddable}) if there exists an embedding $\phi\: X\to\RR^4$.
%
%
Note~that all 2-complexes embed in $\RR^5$, though not necessarily in $\RR^4$.
The classical example for a non-embeddable 2-complex is the triple cone over $K_5$ \cite{grunbaum1969imbeddings}.
Let $\mathcal K_n$ be the complex obtained from the complete graph $K_n$ by attaching a 2-cell along each triangle. 
Then $\mathcal K_7$ provides another example of a non-embeddable 2-complex (\cite[Theorem 5.1.1]{matouvsek2003using}).
Removing a single 2-cell $\Delta$ from $\mathcal K_7$ yields a complex $\mathcal K_7-\Delta$ that is however embeddable, and we will frequently use it as an example:


\begin{example}[$\mathcal K_7-\Delta$]
    \label{ex:K_7-Delta}
    We have $\mathcal K_4\simeq \Sph^2$ and $K_3\simeq \Sph^1$.
    The join of two spheres is a sphere: $K_3\star\mathcal K_4\simeq \Sph^1\star \Sph^2\simeq\Sph^4$ \cite[p.\ 9]{Hatcher}. 
    Since $\mathcal K_7-\Delta\subset K_3\star\mathcal K_4$ (the~missing 2-cell $\Delta$ corresponds precisely to the missing ``filling'' of $K_3$), this provides an embedding of $\mathcal K_7-\Delta$ into $\Sph^4$.
\end{example}

Moreover, we can add back in the missing 2-cell and obtain an ``almost embedding'' of $\mathcal K_7$ with only a single intersection between two 2-cells with disjoint~boundary as follows:

\begin{example}
    \label{ex:K7_almost_embeds}
    Starting from the construction $X:=K_3\star \mathcal K_4$ of \cref{ex:K_7-Delta}, pick a point $x$ in the interior of some 2-face of $\mathcal K_4\subset X$, and note that the cone $C_x$ over the circle $K_3\subset X$ with apex at $x$ lies entirely within $X$.
    Moreover, $C_x$ is homeomorphic to a disc filling $K_3$. 
    By the definition of join \cite[p.\ 9]{Hatcher}, $x$ is the only intersection of $C_x$ with the rest of $\mathcal K_7-\Delta \subset X$.
\end{example}

\subsection{Full complexes and 4-flat graphs}

Throughout the text we use the following definition for 4-flat graphs:

\begin{definition}[full complex, 4-flat graph]\quad
    \label{def:4_flat}
    \label{def:full_complex}
    \begin{myenumerate}
        \item The \emph{full complex} $X(G)$ of $G$ is the CW complex obtained by attaching~a~2-cell along each cycle of $G$.
        \item A graph is \emph{4-flat} if its full complex is embeddable. 
    \end{myenumerate}
\end{definition}

The full complex is clearly a regular complex. 
Since every regular 2-complex with 1-skeleton $G$ is a subcomplex of $X(G)$, this is equivalent to the definition given in the introduction.

Some relevant examples of graphs that are not 4-flat follow from previous discussions: since $\mathcal K_7$ is not~embeddable, $K_7$ is not 4-flat.
Since the triple cone over~$K_5$ is not embeddable, $K_5 * \overline K_3$ is not 4-flat (here $\overline K_3$ denotes the complement graph of $K_3$, and $*$ denotes the graph join operation).
Since $K_7$ is known to be an excluded minor, $K_7-e$ ($K_7$ with an edge removed) is 4-flat.




\subsection{Contraction}
\label{sec:contractions}

Given an embeddable 2-complex $X$ and a cell $c\subseteq \cpx$ with injective attachment map, it is known that the quotient complex $\cpx/c$ is embeddable as well.
To see this, choose a small neighborhood $B$ of $c^\phi$ and replace its interior with a cone over $\partial B\cap X^\phi$. This argument is made precise by the following lemma.


\begin{lemma}
If $X$ is embeddable and $c\subseteq X$ is collapsible\footnote{We will refrain from recalling the definition of collapsible complex, as we will only need the trivial special case where $c$ is a regular 2-cell or an edge with distinct end-vertices in a complex.} (\eg\ when $c$ is a cell of $X$) 
then $X/c$ is embeddable too.
\end{lemma}

The proof requires the theory of regular neighborhoods, see \eg\ \cite[Chapter 3]{rourke2012introduction}.

\begin{proof}[Proof (sketch)]
    %
    Fix a triangulation of $\RR^d$ with simplicial subcomplexes that triangulate $X^\phi$ and $c^\phi$ respectively (which exists by \cite[Addendum 2.12]{rourke2012introduction}). 
    Let $B\subset\RR^4$ be a regular neighborhood of $c^\phi$ with respect to this triangulation (\eg\ as defined in \cite[Chapter 3]{rourke2012introduction}). 
    Since $c^\phi$ is collapsible, $B$ is a ball  (by \cite[Corollary 3.27]{rourke2012introduction}).
    Then there exists a homeomorphism $\smash{\psi\:B\rightarrow K}$ to a convex set $K\subset\RR^4$.
    Let $C$ be the cone 
    over $\psi(\partial B\cap X^\phi)\subset \partial K$ with apex at some interior point of $K$.
    Remove $X^\phi\cap B$ and replace it with $\psi^{-1}(C)$.
    This yields an embedding of $X/c$.
\end{proof}



\tempnewpage

\section{Operations that preserve embeddability}
\label{sec:operations}

In this section we study operations on general (2-dimensional) CW complexes~and circumstances under which they preserve embeddability. 
The eventual goal~is~to~develop a set of flexible tools to build and modify full complexes of 4-flat graphs, but this will also lead to new results and conjectures for general 2-complexes.

In order to be concise, we will often not distinguish an attachment map $\partial c\:\partial \mathbb D^2\to G_X$ from its image, considering $\partial c$ as both a parametrized curve and a subset of $X$, with the understanding that $\partial c$ is injective when we do so. 

\subsection{Joining 2-cells}
\label{sec:joining}

Let $c_1,c_2\subseteq X$ be two distinct 2-cells and  $\gamma\subset\partial c_1\cap \partial c_2$ a common path  (potentially of length zero) in their boundaries.
\emph{Joining} $c_1$ and $c_2$ along $\gamma$ means to replace $c_1,c_2$ by a new 2-cell $c$ whose attachment map is a concatenation of $\partial c_1-\gamma$ and $\partial c_2-\gamma$.  

\begin{figure}[h!]
    \centering
    \includegraphics[width=0.65\textwidth]{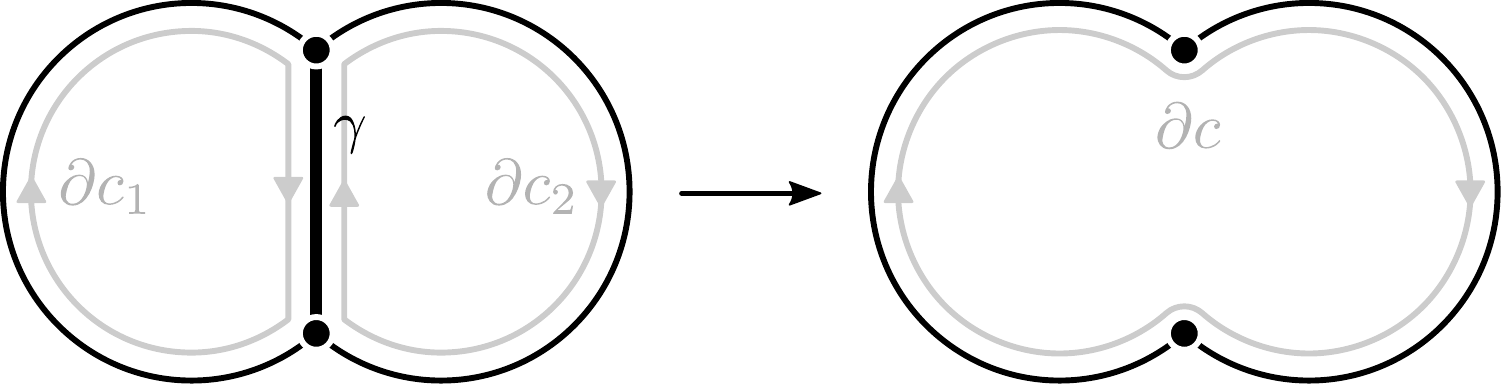}
    \caption{Joining 2-cells across a path $\gamma$.}
    \label{fig:DY_YD}
\end{figure}

\begin{lemma}
    \label{res:joining}
    \label{res:joining_embeds}
    Let $X$ be an embeddable 2-complex, and $c_1,c_2$ two of its 2-cells. Then the 2-complex obtained by joining $c_1,c_2$ along a path $\gamma$ of length $\ell \le 1$ (\ie\ at a vertex or edge) is embeddable too.
\end{lemma}
%
\begin{proof}
  Fix an embedding $\phi\:X\to\RR^4$. 
    Choose an end-vertex $x$ of  $\gamma^\phi$, and a link neighbor\-hood $B\subset\RR^4$ of $x$.
    The idea is to delete $B\cap c_1^\phi$ and $B\cap c_2^\phi$, to fill the resulting holes with suitable  ``patches'' that join the 2-cells into a single one and, in case $\ell=1$, to contract $\gamma^\phi\setminus B$ onto the other end-vertex of~$\gamma$.
    To define the patches, let $\rho_i:=\partial B\cap c_i^\phi$ be the paths in the link that correspond to the original 2-cells.
    We distinguish two cases according to the length $\ell$ of $\gamma$.

    \emph{Case $\ell=0$:} 
    There exists a square $[0,1]^2\simeq Q\subseteq \partial B$ with two opposite edges attached along $\rho_1$ and $\rho_2$ respectively, disjoint from the rest of the link,\nls and mapped into $\RR^4$ non-injectively 
     only where required by the boundary conditions.
    Let~$C_1,C_2\subset B$ be cones over the other two edges of $Q$ with apex at $x$.
    The union $Q\cup C_1\cup C_2$ then patches the hole and joins the 2-cells across $\gamma$ (note that $\gamma^\phi=x$).

    \emph{Case $\ell=1$:}
    Since $x$ is an end-vertex of $\gamma$, the link at~$x$~contains a unique point $y:=\partial B\cap \gamma^\phi$ corresponding to $\gamma$.
    There also exists a (filled in) triangle $\Delta\subseteq \partial B$ with a vertex mapped to $y$ and its two adjacent edges attached along $\rho_1\cup\rho_2$, disjoint from the rest of the link, and mapped into $\RR^4$ non-injective only where required by the boundary conditions.
    Let $C\subset B$ be the cone over the third edge of $\Delta$ with its apex at $x$.
    The union $\Delta\cup C$ patches the hole and joins the 2-cells across the small initial segment $B\cap\gamma^\phi$ of $\gamma^\phi$.
    Finally, we contract $\gamma^\phi\setminus B$ onto the other end-vertex of~$\gamma$ to finalize the joining.
\end{proof}

Joining along a path of length two is no longer guaranteed to preserve embeddabili\-ty.\nspace
An example of this failure will be given in the next section (\cref{ex:joining_fails}).

\subsection{Cloning 2-cells}
\label{sec:cloning}
\label{sec:cloning_cells}

Given a 2-cell $c$ of a 2-complex $X$, \emph{cloning} $c$ means attaching~to~$X$~a~new 2-cell $c'$ with with the same boundary $\partial c'=\partial c$.

\begin{lemma}
    \label{res:cloning}
    \label{res:cloning_embeds}
    Cloning 2-cells preserves embeddability.
\end{lemma}

\begin{proof}
    Let $\phi\:\cpx\to\RR^4$ be an embedding, and $\mathcal T$ a triangulation of $c^\phi$ such that each 2-cell of $\mathcal T$ coincides with the convex hull of a triple of points.

    If $\mathcal T$ consists of a single triangle $\Delta$ then a clone of $c$ can be embedded as follows: let $\sigma\subset\RR^4$ be a 3-dimensional simplex with $\sigma\cap X^\phi=\partial \sigma\cap X^\phi=\Delta$.
    Then $\partial\sigma- \Delta$ is an embedding of the clone $c'$.
    
    If $\mathcal T$ contains more than one triangle then let $G$ be the 1-skeleton of $\mathcal T$ and let $G'$ be the subgraph induced by $V(G)\setminus\partial c^\phi$.
    Let $F$ be a spanning forest of $G'$~and,\nls for each connected component $\tau\subseteq F$, choose an edge $e_\tau\in E(G)$ that~connects~$\tau$~to~a boundary vertex $v_{\tau}\in \partial c^\phi$.
    Now clone each triangle of $\mathcal T$ as above, and join the clones along~their shared edges that are not in $F$ using Lemma~\ref{res:joining}.
    Observe that $\Int(c^\phi)\setminus F$ is an open disk and that the above process created a clone of it. 
    Thus, contracting each connected component $\tau\subseteq F$ onto the corresponding boundary vertex $v_\tau\in\partial c^\phi$ turns these disks into embeddings of the 2-cell $c$ and a clone $c'$.
\end{proof}  

Let us point out some subtleties of this result.
First, the proof of \cref{res:cloning} makes~sig\-nificant use of the \PL\ structure and it is not clear how to translate it to, say,\nspace the~to\-pological category.
Second, to embed the clone our construction modifies the initial 2-cell, and it is not clear whether this can be avoided:

\begin{question}
   Let $X$ be an embedded 2-complex, and $c$ one of its 2-cells. Can a clone of $c$ always be embedded without modifying the original embedding of $X$?
\end{question}

For a smooth embedding one can always embed a clone in the normal bundle of a 2-cell without modifying the original.
In contrast, there are topological embeddings of disks that cannot be completed to embeddings of spheres \cite{daverman2006absence,daverman1973scarcity}, which shows that in the topological category cloning 2-cells will require modification (see also~the~dis\-cussion in \cite{MathOverflow406829}).
We are not aware of an answer in the \PL\ category.

Cloning is a powerful tool when combined 
with other operations on 2-cells such as joining:
one can first clone the involved 2-cells and then perform the other operation on the clones.
We will make use of this idea many times.
For example, we will now use it to show that joining 2-cells along a path of length two is not~guaranteed to preserve embeddability: 

\begin{example}
    \label{ex:joining}
    \label{ex:joining_fails}
    Recall that $\mathcal K_7-\Delta$ is embeddable  (\cref{ex:K_7-Delta}). We label its vertices as $x_1,...,x_7$, so that $x_1 x_2 x_3$ is the missing 2-cell. Clone the 2-cells $x_1x_2x_4$ and $x_2x_3x_4$, and join the clones along the shared edge $x_2x_4$. This creates a 2-cell $c$ bounded by $x_1x_2x_3x_4$, and preserves~embeddability. However, 
    joining $c$ and a clone of the 2-cell bounded by $x_3x_4x_1$ along the shared 2-path $x_3x_4x_1$ introduces a new 2-cell bounded by $x_1x_2x_3$. 
    We thereby created the~mis\-sing triangle, and thus the non-embeddable complex $\mathcal K_7$. 
\end{example}

\subsection{Application: complexes with few dominating vertices}
\label{sec:dominating_vertices}

De la Harpe \cite{MathOverflow19618} asked whether every 2-complex with a single vertex embeds in $\RR^4$, and an answer was found involving a difficult theorem of Stallings. The follo\-wing result strengthens this.
Its elementary proof was inspired by a discussion with \Tam. 

\begin{theorem} 
\label{thm vw}
\label{thm ops}
\label{res:all_2_cells_vw}
\label{res:2_dominating_vertices}
If there are two vertices $v,w\in \cpx$ so that each 2-cell $c\subseteq \cpx$~is~incident to at least one of them, then $\cpx$ can be embedded in $\RR^4$.
\end{theorem}
\begin{proof}
    Assume first that all 2-cells are incident to the vertex $v\in \cpx$.
    Fix an~embedding $\phi$ of the 1-skeleton $\skel\cpx$ into $\RR^3\times\{0\}\subset\RR^4$. 
    Let $C_G\subset \RR^3\times\RR_+$ be a cone over $G_X^{\smash{\phi}}$. 
    For each vertex $u$ and edge $e$ of $G$, let $C_u,C_e\subseteq C_G$ 
    denote the subcone over $u$ and $e$ respectively. 
    Then $C_G$ is a 2-complex with each $C_u, u\in V(G)$ an edge of $C_G$, and each $C_e, e\in E(G)$ a 2-cell.
    To~em\-bed a 2-cell $c\subseteq X$, we clone all 2-cells $C_e$, $e\subset\partial c$ 
    and join the clones along the edges $C_u$, $u\in\partial c\setminus v$ using \cref{res:joining_embeds}. 
    This results in a single 2-cell $\tilde c$ with boundary $\partial \tilde c=\partial c\cup C_v$.
    Hav\-ing~constructed $\tilde c$ for all 2-cells $c\subseteq \cpx$, we contract the edge $C_v$ onto $v$.
    This results in~an embedding of $X$.
    
    If each 2-cell is incident to one of the two vertices $v,w\in \cpx$, we proceed as above, except that we also embed a second cone $C_G'$ into $\RR^3\times\RR_-$. 
    We use $C_G$ to construct the 2-cells that contain $v$ 
    as above, and then use $C_G'$ to construct the remaining 2-cells, which contain $w$. 
\end{proof}
%



\begin{remark}
    If every 2-cell is incident to one of three vertices, then embeddability is no longer guaranteed. A  counterexample is the triple-cone over $K_5$ \cite{grunbaum1969imbeddings}. 
\end{remark}

\subsection{Rerouting 2-cells}
\label{sec:rerouting}

Let $\gamma_1,\gamma_2\subseteq G$ be \emph{parallel} paths, that is, with the same end-vertices.
By \emph{rerouting} a 2-cell $c$ with $\gamma_1\subset \partial c$ from $\gamma_1$ to $\gamma_2$ we mean replacing $c$ by a 2-cell $c'$  attached along $(\partial c -\gamma_1)\cup\gamma_2$ instead.




\begin{lemma}
    \label{res:rerouting}
    \label{res:rerouting_embeds}
    If $\gamma_1,\gamma_2$ are parallel paths and $\gamma_1$ is of length $\le 1$ (\ie\ it is a vertex or edge), then the following are equivalent:
    \begin{myenumerate}
        \item Rerouting a 2-cell from $\gamma_1$ to $\gamma_2$ yields an embeddable complex. 
        \item Attaching a 2-cell along $\gamma_1\cup\gamma_2$ yields an embeddable complex. 
    \end{myenumerate}
\end{lemma}
\begin{proof}

    We start with the implication \itmto12. Easily, attaching a new 2-cell $\tilde c$ along $\partial\tilde c=\gamma_1$ (if $\gamma_1$ is~an~edge then this means $\partial \tilde c$ tra\-verses $\gamma_1$ twice, once in each direction) preserves embedda\-bility. 
    Since rerouting from $\gamma_1$ to $\gamma_2$ preserved embeddability, we can reroute~a part of $\partial \tilde c$ attached along $\gamma_1$ to traverse $\gamma_2$ instead to create a 2-cell along $\gamma_1\cup\gamma_2$.
    
    To prove \itmto21,  let $\tilde c$ be a 2-cell attached along $\partial c=\gamma_1\cup\gamma_2$. 
    To reroute a 2-cell $c$ from $\gamma_1$ to $\gamma_2$, clone $\tilde c$ using \cref{res:cloning_embeds}, and join the clone with $c$ at $\gamma_1$. Since $\gamma_1$ is of length $\le 1$, this preserves embeddability by \cref{res:joining}.
\end{proof}

We remark that rerouting from a path of length two, even with a 2-cell along $\gamma_1\cup\gamma_2$, is \emph{not} guaranteed to preserve embeddability: 

\begin{example}
    \label{ex:rerouting}
    In $\mathcal K_7-\Delta$, with the notation of \cref{ex:joining_fails}, clone the 2-cells $x_1x_2x_4$ and $x_2x_3x_4$, and join the clones at the shared edge $x_2x_4$. This creates a 2-cell $c$ along $x_1x_2x_3x_4$ and preserves~embeddability.
    The paths $\gamma_1=x_1x_4x_3$ and~$\gamma_2=x_1x_2x_3$ now satisfy \cref{res:rerouting}~\itm2~(as~$c$ is attached along $\gamma_1\cup\gamma_2$).
    However, cloning the 2-cell along $x_1 x_3 x_4$ and rerouting the clone from $\gamma_1$ to $\gamma_2$ yields a 2-cell along $x_1x_2x_3$.
    We thereby recreated~the~mis\-sing triangle and thus the non-embeddable complex $\mathcal K_7$. 
\end{example}





\subsection{Collapsing 2-cells}
\label{sec:collapsing}


Given a 2-cell $c\subseteq X$, fix a decomposition $\partial c=\gamma_1\cup\gamma_2$ into parallel paths, as well as a homeomorphism $f\:\gamma_1\to\gamma_2$ that fixes end points.
\emph{Collapsing} a 2-cell $c$ is the operation of deleting $c$ and identifying the paths $\gamma_1$ and $\gamma_2$ along the homeomorphism $f$.

Collapsing can be interpreted as an extreme form of rerouting: we reroute every 2-cell from $\gamma_1$ to $\gamma_2$, even cells that are attached to only a part of $\gamma_1$. In contrast to rerouting single 2-cells, where embeddability is preserved only in few special cases, collapsing is always possible:

\begin{lemma}
    \label{res:collapsing}
    Collapsing 2-cells preserves embeddability.
\end{lemma}
\begin{proof}    
    Let $v_1,...,v_n\in \gamma_1$ be an enumeration of the vertices of $\gamma_1$ and let $e_i\subset\gamma_1$ be the edge between $v_i$ and $v_{i+1}$.
    We subdivide $c$ into a ``chain of 2-cells''~$c_2,...,c_{n-1}\subseteq c$ by introducing new edges $\tilde e_i\subset c$ that connect $v_i$ and $f(v_i)$.
    Then $\partial c_i = e_i\cup \tilde e_i \cup f(e_i)\cup \tilde e_{i-1}$.
    If we contract all $\tilde e_i$ then we have $\partial c_i=e_i\cup f(e_i)$.
    Using \cref{res:rerouting} \itmto21 we can reroute all 2-cells attached along $e_i$ to $f(e_i)$. 
    Finally, we delete $c$ including all $e_i$.
\end{proof}

For later use we also introduce collapsing cylinders: given a graph $H$, collapsing a cylinder $H\times[0,1]\subseteq X$ means deleting it and~identifying $H\times\{0\}$ with $H\times\{1\}$ in the obvious way.

\begin{corollary}
    \label{res:collapsing_cylinder}
    Collapsing cylinders preserves embeddability.
\end{corollary}
\begin{proof}
   For each vertex $v\in H$ contract $v\times [0,1]$.
   This turns $H\times[0,1]$ into a number of 2-cells $c_e$ indexed by edges $e\subseteq H$ and bounded by $\partial c_e=(e\times\{0\})\cup(e\times\{1\})$.
    Subsequently collapse each 2-cell $c_e$ onto $e\times\{0\}\subset\partial c_e$.
\end{proof}

\subsection{Cloning and merging edges}
\label{sec:merging_edges}
\label{sec:cloning_edges}

We now define an edge-cloning operation, which comes in two variants, one for graphs, and one for 2-complexes.
By \emph{cloning} an edge $e\subseteq G$ of a graph we mean adding a new edge $e'$ parallel to $e$.
In a 2-complex $X$, \emph{cloning} $e$ clones $e$ in $G_X$, and additionally adds the following new 2-cells:
\begin{myenumerate}
    \item a new 2-cell $\tilde c$ with boundary $e\cup e'$, and
    \item for each 2-cell $c\subseteq X$ incident with  $e$, another 2-cell $c'$ with the same~attachment map as $c$ except that it traverses $e'$ instead of $e$.
\end{myenumerate}



\begin{corollary}
    \label{res:cloning_edges}
    Cloning an edge $e$ of a 2-complex preserves embeddability.
\end{corollary}
%
%
\begin{proof}
    Fix an embedding $\phi\: X\to\RR^4$.
    First, embed a disk $D\subset\RR^4$ with $D\cap X^\phi=\partial D\cap X^\phi=e^\phi$. 
    We can now interpret $\partial D- e^\phi$ as an embedding of $e'$ and $D$ as an embedding of a 2-cell $\tilde c$ with $\partial \tilde c=e\cup e'$. 
    For each 2-cell $c$ incident with $e$, we do the following: we clone $c$, and reroute the clone from $e$ to $e'$. This preserves embeddability by \cref{res:cloning,res:rerouting}.
\end{proof}

We can now confirm van der Holst's \cref{conj:cloning_edges} on 4-flat graphs:

\begin{corollary}
    \label{res:cloning_edges_4_flat}
    Cloning an edge $e$ of a graph preserves 4-flatness.
\end{corollary}
\begin{proof}
    Let $G'$ be the graph obtained from $G$ by cloning the edge $e$.
    Let $X'$ be the complex obtained from $\Xany(G)$ by cloning the edge $e$.
    Observe that $\Xany(G')=X'$.\nls
    The claim then follows from \cref{res:cloning_edges}.
\end{proof}

We now consider the reverse operation of cloning: by \emph{merging} a pair of parallel edges $e_1,e_2\subseteq X$ we mean to identify them in the obvious way.

\begin{corollary}
    \label{res:merging}
    \label{res:merging_iff_filling}
    Given an embeddable 2-complex $X$ with parallel edges \mbox{$e_1,e_2\subset X$}, the following are equivalent: 
    \begin{myenumerate}
        \item
        The complex $X'$ obtained by merging $e_1$ and $e_2$ is embeddable.
        \item 
        The complex $X''$ obtained by attaching a 2-cell along $e_1\cup e_2$ is embeddable.
    \end{myenumerate}
\end{corollary}
\begin{proof}
    Let $e$ denote  the  edge of $X'$ resulting from merging.
    Then the complex obtained from $X'$~by cloning $e$ contains $X''$ as a subcomplex. 
    Therefore \itmto12 follows from \cref{res:cloning_edges}.
    Conversely, $X'$ is obtained from $X''$ by collapsing the 2-cell along $e_1\cup e_2$, onto  $e_1$ say. Therefore \itmto21 follows from \cref{res:collapsing}.
    %
    %
\end{proof}



\subsection{Stellifying cycles and $\boldsymbol{\DY}$-\trafos}
\label{sec:stellifying_cycles}
\label{sec:DY_trafos}

\Cref{res:merging_iff_filling} states that, from the point of view of embeddability, ``collapsing a 2-cycle'' is the same as~filling~it~by~a 2-cell.
In this section we investigate whether statements of this sort can be extended to longer cycles.

Given a cycle $C\subseteq G$ of length $\ell\ge 3$, the operation of \emph{stellifying} $C$ consists in replacing $C$ by~a~star (see \cref{fig:stellification}).
More precisely, let $C$ have vertices $v_1,...,v_\ell$ and edges $e_{ij}=v_iv_j$; to stellify~$C$~we remove the edges $e_{ij}$ and then add a new vertex $w$ as well as the edges $e_i:=w v_i$.
If $C\subseteq G_X$ is part of a 2-complex, then \emph{stellifying} additionally reroutes every 2-cell that traverses $e_{ij}$ to run along $e_i\cup e_j$ instead.

Let $\Gcirc$ be a graph containing a cycle $C$, and let $\Gstar$ be the graph obtained by stellifying $C$ in $\Gcirc$.
Similarly, let $\Xcirc$ be a 2-complex with 1-skeleton $\Gcirc$ that contains no 2-cell bounded by $C$, and let~$\Xstar$ be the complex obtained by stellifying $C$ in $\Xcirc$. Finally, let $\Xdisk$ be the complex obtained from $\Xcirc$ by attaching a 2-cell injectively along $C$.

\begin{figure}
    \centering
    \includegraphics[width=0.6\textwidth]{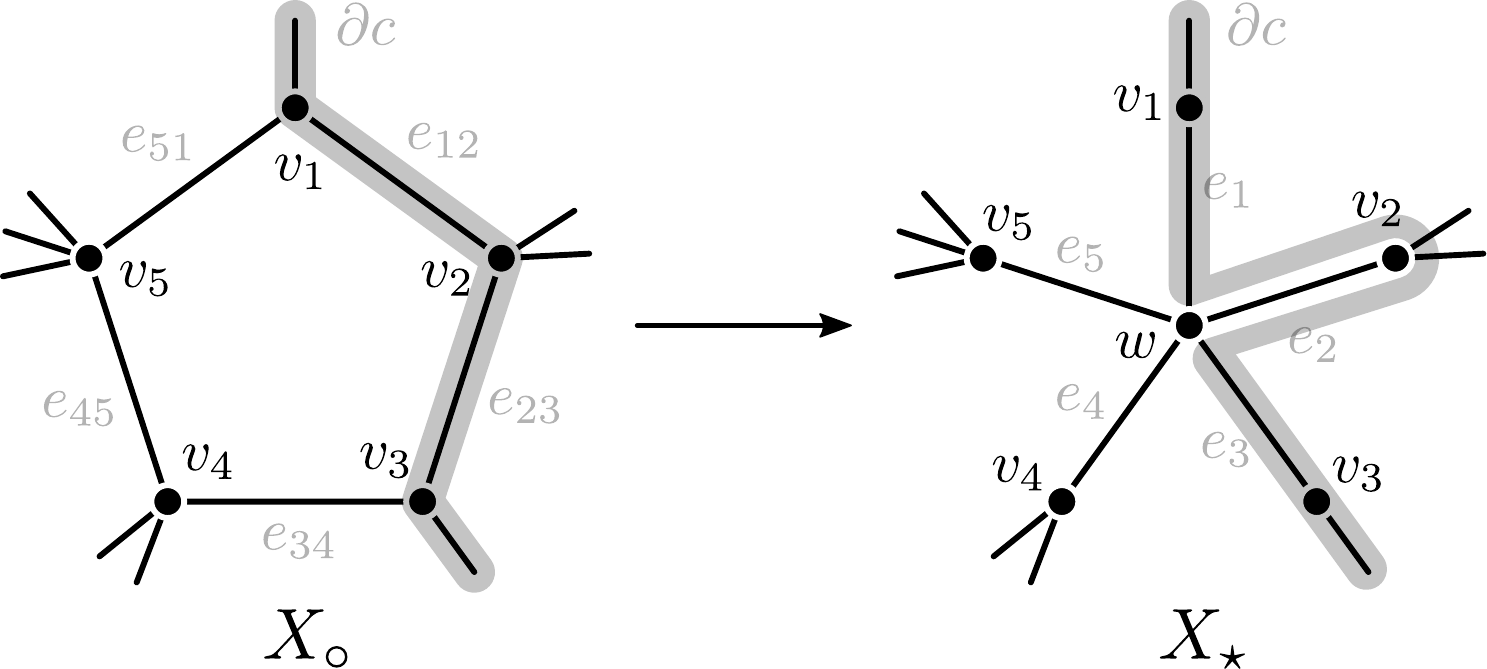}
    \caption{Visualization of stellification of a 5-cycle in a 2-complex. The highlighted paths show how an attachment map $\partial c$ gets modified in the process.}
    \label{fig:stellification}
\end{figure}



Generalizing \cref{res:merging} to longer cycles turns out to be non-trivial. The following conjecture goes in this direction:

\begin{conjecture}
    \label{conj:stellifying}
    If $\Xcirc$ is embeddable, then the following are equivalent:
    %
    \begin{myenumerate}
        \item $\Xstar$ is embeddable.
        \item $\Xdisk$ is embeddable.
    \end{myenumerate}
\end{conjecture}

\begin{lemma}
    \label{res:stellifying}
    \label{res:stellifying_ii_to_i}
    The implication  \itmto21 of \cref{conj:stellifying} holds.
\end{lemma}
\begin{proof}
    Let $c\subseteq \Xdisk$ be the 2-cell with $\partial c=C$. We subdivide $c$ by adding a vertex $w$ in its interior as well as edges $e_i$ connecting $w$ to $v_i$.
    Let $c_i\subset c$ be the sub-cell~with $\partial c_i=e_i\cup e_{ij}\cup e_j$. Stellifying $C$ is then achieved by collapsing each $c_i$ onto $e_i\cup e_j$.
    This preserves embeddability by \cref{res:collapsing}.
\end{proof}

So far we are unable to prove the other direction \itmto12.
We also emphasize an important difference to the 2-cycle version (\cref{res:merging}): in \cref{conj:stellifying} we need to assume the embeddability of $\Xcirc$, as otherwise the direction \itmto12 is not true.
An example of this failure is given in \cref{sec:YD_counterexample}.
For a potential~coun\-terexample to \cref{conj:stellifying} see \cref{q:J3n}.



A particularly common instance of stellification is $\ell=3$, in which case the~oper\-ation is known as a \emph{\DYtrafo}. Using this terminology, \cref{res:stellifying} reads

\begin{corollary}
    \label{res:DY_YD}
    Applying a \DYtrafo\ on a 2-complex preserves embeddability.
\end{corollary}

We next work towards van der Holst's \cref{conj:DY_YD} for 4-flat graphs. 
Suppose now that the cycle $C$ in the definition of $\Gcirc$ is a triangle, and let~us~use the notation $G_\Delta$ and $G_\Y$ (and $X_\Delta$ and $X_\Y$ accordingly) instead of $\Gcirc$ and $\Gstar$ to emphasize this.
We make the following observation, which follows from the definitions:

\begin{observation}
    \label{res:XY_eq_Xany_of_GY}
    If $X_\Delta:=\Xany(G_\Delta)$, then $X_\Y=\Xany(G_\Y)$.
    \end{observation}

This statement~is~ more subtle than one may think at first sight.
 Indeed, the analogous statement is incorrect for  $|C| \ge 4$:
    while in $\Xany(\Gstar)$ there is a 2-cell that traverses $e_1\cup e_3$, there is no such 2-cell in $\Xstar$.

The \DY-part of van der Holst's \cref{conj:DY_YD} follows immediately:


\begin{corollary}
    \label{res:DY_YD_4_flat}
    Performing a \DYtrafo\ on a graph preserves 4-flatness.
\end{corollary}
\begin{proof}
    Set $X_\Delta:=\Xany(G_\Delta)$. By  \cref{res:XY_eq_Xany_of_GY} we have $X_\Y=\Xany(G_\Y)$.
    The claim then follows from \cref{res:DY_YD}.
\end{proof}





\begin{observation}\quad
    \label{res:stellifying_observations}
    \begin{myenumerate}
        \item By cloning all edges of the cycle $C$ we create a second cycle $C'$ on the~same vertex set. Stellifying this new cycle has the net effect of adding to $G$~a~suspension vertex over $C$ while also preserving the original cycle. 
        \item Since both cloning edges and \DYtrafos\ preserve 4-flatness, \itm1 implies that adding a suspension vertex over a triangle in $G$ preserves 4-flatness. 
        We prove a strengthening of this in \cref{res:3_clique_sums}.
        \item 
        Adding a cone over a cycle and then deleting all but two of the cone edges has the net effect of adding a chord to the cycle. Any chord can be created in this way.
        \item We can now deduce that stellifying a cycle of length $\ell\ge 4$ does not preserve~4-flatness: via cloning edges and stellifying 4-cycles one can reconstruct~the missing edge in the 4-flat graph $K_7-e$, thereby turning it into $K_7$, which is not 4-flat.
    \end{myenumerate}
\end{observation}
%


    

Consider the following generalized stellification operation: stellifying a subgraph $H\subseteq G$ means to delete its edges and to add a new suspension vertex over it (and note that with the help of doubling edges, deleting the edges of $H$ does not make~any difference for 4-flatness).
Following the reasoning of \cref{res:stellifying_observations}~\itm3 and \itm4, stellifying a non-complete subgraph cannot preserve 4-flatness.
Likewise, stellifying $K_5$ cannot preserve 4-flatness either, because we can use it to turn $K_5$ into a triple suspension over $K_5$, which is not 4-flat.
This motivates the following question:

\begin{question}
    \label{q:stellifying_K4}
    Does stellifying a $K_4$ preserve 4-flatness?
\end{question}

\tempnewpage

\section{Reverse stellification and $\YD$-\trafos}
\label{sec:YD}

In \cref{sec:stellifying_cycles} we explored the conditions under which the embeddability~of~$\Xcirc$ implies the embeddability of $\Xstar$, as well as analogous questions for 4-flatness. 
In~this section we ask about the opposite direction: if $\Xstar$ embeds, must $\Xcirc$ embed to?\nls
This includes the \YD-part of van der Holst's \cref{conj:DY_YD}:

\begin{conjecture}
    Performing a \YDtrafo\ on a graph preserves 4-flatness.
\end{conjecture}

This conjecture  remains open.
We do however present evidence that its statement is  at the boundary of what can be true:
\begin{myenumerate}
    \item an analogous statement is not true for cycles of length $\ell\ge 4$: even if $\Gstar$ is 4-flat, $\Gcirc$ might not be (see \cref{ex:reverse_stellifying_4_cycle_4_flat} below).
    \item an analogous statement is not true for embeddability of general 2-complexes: even if $X_\Y$ is embeddable, $X_\Delta$ might not be (\cref{sec:YD_counterexample}).
\end{myenumerate}

\begin{example}
    \label{ex:reverse_stellifying_4_cycle_4_flat}    
    $\Gcirc:=K_7$ is not 4-flat.
    However, stellifying any of its 4-cycles~results in a 4-flat graph $\Gstar$.
    To see this, we make use of the fact that the double suspension of a planar graph is 4-flat (we prove this in \cref{res:linkless_planar_outerplanar} in the next section).
    Since we can delete two vertices of $\Gstar$ and obtain a planar graph (see~\cref{fig:star_circ_counterexample}),\nls $\Gstar$~is~contained in such a double suspension.
    In conclusion, reverse stellification at a 4-cycle does not always preserve 4-flatness.
\end{example}

\begin{figure}[h!]
    \centering
    \includegraphics[width=0.95\textwidth]{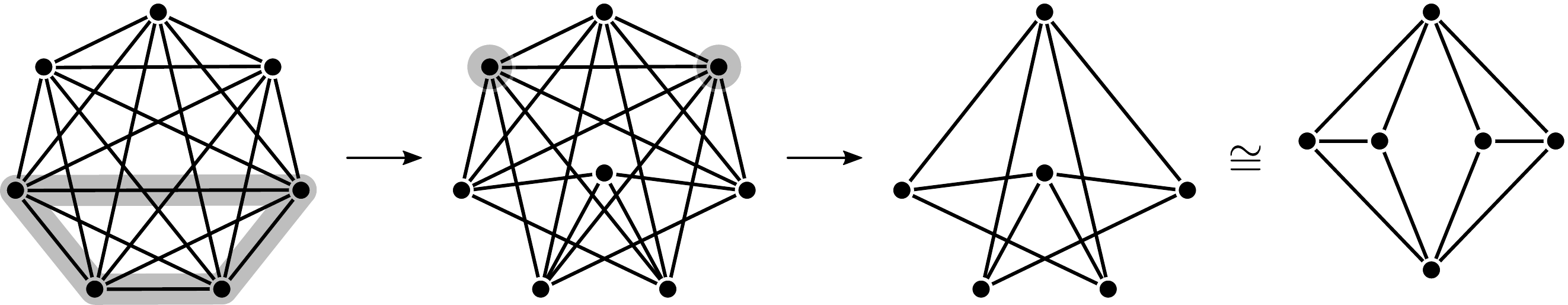}
    \caption{Starting from $K_7$ we first stellify a 4-cycle, then we delete~two vertices to obtain a planar graph.}
    \label{fig:star_circ_counterexample}
\end{figure}


\subsection{A \YDtrafo\ that does not preserve embeddability}
\label{sec:YD_counterexample}



The following construction is based on the Freedman-Krushkal-Teichner (FKT) complex \FKT\ \cite{freedman1994van}.
One of the main characteristics of the FKT complex is that it does \emph{not} embed into $\RR^4$, 
despite its van Kampen obstruction being zero.

We shall describe two 2-complexes, $X_\Delta$ and $X_\Y$, the latter being the \DY-transfor\-ma\-tion of the former. We will show that $X_\Y$ is embeddable, while $X_\Delta$~is~not~(by~proving that any potential embedding of it can be turned into an embedding of the FKT complex). 

We briefly recall the FKT complex and then modify it to obtain $X_\Delta$.

\subsubsection{The Freedman-Krushkal-Teichner complex}
The complex is built as follows:
\begin{enumerate}
    \item 
    start from two copies of $K_7$, say, $K_7$ with vertices $x_1,...,x_7$, and $\tilde K_7$ with vertices $\tilde x_1,...,\tilde x_7$.
    
    \item 
    \label{itm:all_2cells_but_one}
    attach a 2-cell $c_{ijk}$ to each 3-cycle $x_i x_j x_k$ of $K_7$, except $x_4x_5x_6$.
    And~analogously, attach a 2-cell $\tilde c_{ijk}$ to each 3-cycle $\tilde x_i\tilde x_j\tilde x_k$ of $\tilde K_7$, except $\tilde x_4\tilde x_5\tilde x_6$.
    
    \label{itm:K6_K6}

    \item 
    \label{itm:e6_e7}
    add an edge $e_6$ between $x_6$ and $\tilde x_6$. 


    \item attach a 2-cell $c^*$ along the closed walk $\partial c^*$ given by 
    \begin{align}\label{eq:commutator}
        x_6\, x_5\, x_4\, x_6\, \tilde x_6\, \tilde x_5\, \tilde x_4\, \tilde x_6\, x_6\, x_4\, x_5\, x_6\, \tilde x_6\,  \tilde x_4\, \tilde x_5\, \tilde x_6\, x_6.
    \end{align}
\end{enumerate}





The contrived attachment map $\partial c^*$ is chosen specifically to force the emergence of a Borro\-mean-ring-like structure in every potential embedding of the complex (details can be found in the original work of Freedman, Krushkal and Teichner \cite{freedman1994van}, or with more explanations in \cite{avramidi2021fungible}). The relevant property of $\FKT$ for our purposes is the following (see \cite[Section 3.2]{freedman1994van}): no embedding of $\FKT -c^*$ can be extended to a mapping of $\FKT$ into $\RR^4$, which means that we are allowing self-intersections of $c^*$, but still no intersections between $c^*$ and other parts of $\FKT$. 

\subsubsection{The complexes $X_\Delta$ and $X_\Y$}
We modify \FKT\ to obtain the complex $X_\Delta$ as follows: 
\begin{enumerate}
    \item identify $c_{123}$ 
     and $\tilde c_{123}$ in an interior point; call the resulting point $w$.
    \item identify $c^*$ and $ c_{123}$ in an interior point; call the resulting point $v$.
    \item identify $c^*$ and $\tilde c_{123}$ in an interior point; call the resulting point $\tilde v$.
\end{enumerate}

These three identifications do not respect the CW complex structure, which needs to be restored using suitable subdivisions.
We choose subdivisions which in particular contain the following edges:
%
\begin{enumerate}[resume]
    \item an edge $e_{123}\subset \Int(c_{123})$ from $v$ to $w$.
    \item an edge  $\tilde e_{123}\subset \Int(\tilde c_{123})$ from $\tilde v$ to $w$.
    \item an edge $e^*\subset \Int(c^*)$ from $v$ to $\tilde v$.
\end{enumerate}
The edges $e_{123},\tilde e_{123}$ and $e^*$ form the triangle $\Delta$ on which we later perform the~\DYtrafo.
The complex obtained by this \DYtrafo\ is $X_\Y$.

\subsubsection{Non-embeddability of $X_\Delta$}
\label{sec:XD_embeddable}

In \cite[Section 3.2]{freedman1994van} 
the authors prove 
that an embedding of $\FKT-c^*$ cannot be extended to a mapping of $\FKT$ into $\RR^4$, in~particular, allowing for self-intersections of $c^*$ (a fact that we will make use of).\nls
%
By~construction, an embedding of $X_\Delta$ defines a mapping of $X_{\mathrm{FKT}}$ into $\RR^4$ in which $c_{123}$, $\tilde c_{123}$ and $c^*$ have pairwise intersections at $v,\tilde v$ and $w$ respectively, but that is an embedding otherwise.
We show that we can get rid of these intersections,\nls creating self-intersections at most for $c^*$, and in this way, turning an embedding of $X_\Delta$ into a mapping of $X_{\mathrm{FKT}}$ that is known to not exist.

\begin{lemma}\label{res:remove_intersection}
Let $\phi_1,\phi_2\:\Sph^2\to\RR^4$ be \PL\ maps, potentially non-injective, with a single intersection $x:=\phi_1\cap \phi_2$, and injective in a neighborhood $B$ of $x$.
Then the intersection can be removed by modifying the $\phi_i$ in this neighborhood of $x$. 
If $\phi_i$ was injective to begin with, then it is still injective after the modification.
\end{lemma}

\begin{proof}
    Delete $(\phi_1\cup \phi_2)\cap B$ and replace it by a copy of the exterior $(\phi_1\cup\phi_2)\setminus B$ via inversion on the 3-sphere $\partial B$.
    This replaces the disk $\phi_i\cap B$ by an image of the disk $\phi_i\setminus B$, keeping $\phi_i$ a map of the 2-sphere, and injective if $\phi_i$ was injective outside of $B$.
    Since there are no intersections of $\phi_1$ and $\phi_2$ outside of $B$, this removed the intersection.
\end{proof}

Observe that the subcomplex $\sigma$ induced on the vertices $x_1,x_2,x_3, x_7\in X_\Delta$ has 2-cells along each triangle (like a 3-simplex) and thus  forms a 2-sphere. 
Analogously, the subcomplex $\tilde\sigma$ on $\tilde x_1,\tilde x_2,\tilde x_3,\tilde x_7$ forms a 2-sphere.
In $X_\Delta$ these 2-spheres intersect exactly once in $w=c_{123}\cap \tilde c_{123}$.
In the embedding this intersection can~then~be~re\-moved using \cref{res:remove_intersection}.

Next, we observe that the boundary curve $\partial c^*$ can be filled in by a (self-intersec\-ting) disk disjoint from $\sigma$: use $\tilde\sigma - \tilde c_{123}$ to fill in the parts $\tilde x_6 \tilde x_5\tilde x_4$ and $\tilde x_6 \tilde x_4\tilde x_5$; the rest can be filled in by a ``collapsed disk'' (see \cref{fig:16_cycle}).
\begin{figure}[h!]
    \centering
    \includegraphics[width=0.5\textwidth]{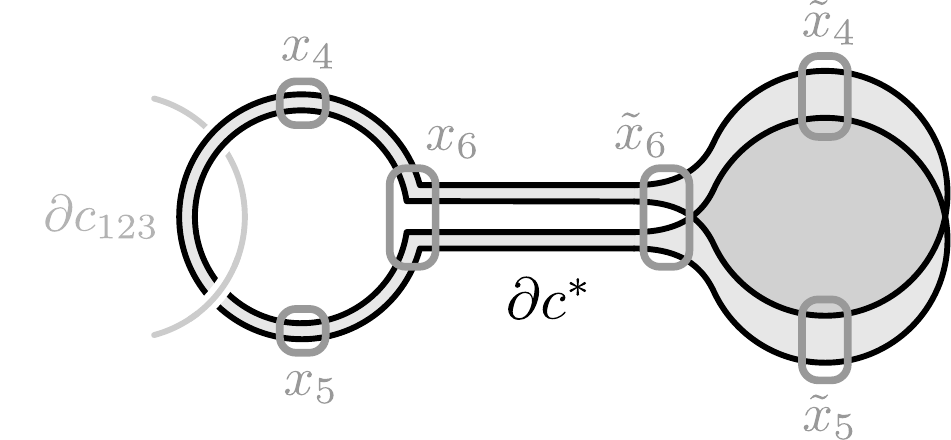}
    \caption{Visualization of a disk attached along $\partial c^*$ as constructed~in \cref{sec:XD_embeddable}. 
    Note that $\partial c^*$ is non-injective but is depicted in an~injec\-tive way to assist the visualization.
    The curve passes through a single vertex in the regions marked in gray.
    }
    \label{fig:16_cycle}
\end{figure}
%
Together with the~embedding of $c^*$ this is a mapping (but not an embedding) of a 2-sphere. 
We now use \cref{res:remove_intersection} to get rid of the unique intersection of $c^*$ and $c_{123}$ at $v$, potentially creating \mbox{self-inter}\-sections of $c^*$.
Analogously, we can get rid of the single intersection between~$c^*$ and $\tilde c_{123}$ at $\tilde v$.

Thus we have obtained a mapping of $\FKT$ into $\RR^4$, having self-intersections~only of $c^*$ and being an embedding otherwise.
This is a contradiction.

\subsubsection{Embeddability of $X_\Y$}


Recall that the complex 
$\mathcal K_7$
is obtained from the complete graph $K_7$ by attaching a 2-cell along each triangle. 
Let $x_1,...,x_7$ be its vertices and let $c_{ijk}$ 
 denote the 2-cell attached along $x_i x_j x_k$.
Recall further that while $\mathcal K_7$ does not embed, there exists a mapping $\phi\:\mathcal K_7\to\RR^4$ that is injective except for a single intersection between, say, $c_{123}$ and $c_{456}$ (\cref{ex:K7_almost_embeds}).
Define the complex~$\mathcal K_7^*$ from $\mathcal K_7$ by identifying $c_{123}$ and $c_{456}$ in an interior point.
We denote the~point of intersection by $y:=c_{123}\cap c_{456}$ and subdivide the intersecting 2-cells by edges of the form $yx_i$ to restore the CW complex structure.
We continue to denote the subdivided 2-cells by $c_{123}$ and $c_{456}$ respectively.
We can now view $\phi$ as an embedding of $\mathcal K_7^*$.
Observe that the link at $y^\phi$ consists of two linked cycles.
We now modify $\mathcal K_7^*$ into $X_\Y$ by a series of operations that preserve embeddability (\cf\ \cref{sec:operations}).

Choose a 4-ball $B\subset\RR^4$ that intersects $(\mathcal K_7^*)^\phi$ only in $y$. By inversion on $\partial B$ we embed a copy $\tilde{\mathcal K}_7^*$ that shares with $\mathcal K_7^*$ only the vertex $y$.
We denote its vertices~by $\tilde x_1,...,\tilde x_7$ and its 2-cells by $\tilde c_{ijk}$ respectively.
Observe that the link at $y^\phi$ now consists of four cycles that belong to two linked pairs but are otherwise unlinked.


Embed a disk $D\subset\RR^4$ with a path in its boundary $\partial D$ attached along $x_6 y x_6$. 
The opposite path $\partial D\setminus(x_6 y \tilde x_6)^\phi$ we now consider as an embedding of an edge~$e_6=x_6\tilde x_6$ and $D$ as a 2-cell attached along $y x_6 \cup y \tilde x_6 \cup e_6$.
After these additions, the link at $y^\phi$ now looks as shown in \cref{fig:y_link_1}.

\begin{figure}[h!]
    \centering
    \includegraphics[width=0.55\textwidth]{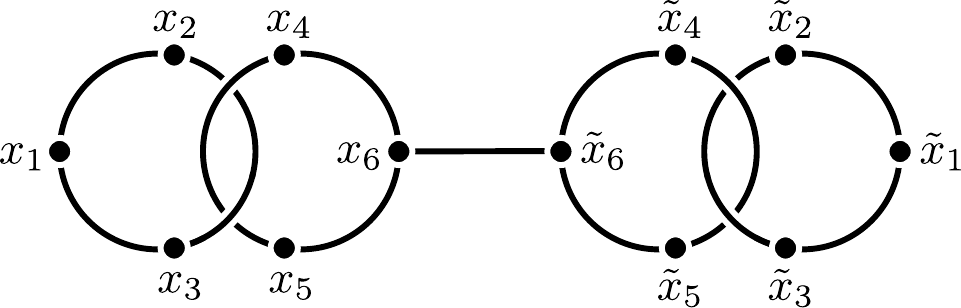}
    \caption{The link at $y^\phi$ at a stage during the construction of $X_\Y$.\nls The point labelled $x_i$ (\resp\ $\tilde x_i$) is the vertex of the link that corresponds to the edge $yx_i$ (\resp\ $y\tilde x_i$) in the complex.}
    \label{fig:y_link_1}
\end{figure}


By cloning relevant 2-cells incident to $y$ (\cf\ \cref{res:cloning_embeds}) and joining the clones suitably at shared edges (\cf\ \cref{res:joining}) we create an embedding of a (subdivided) 2-cell attached along the path $\partial c^*$ (see also \eqref{eq:commutator}) that intersects $c_{123}$ and $\tilde c_{123}$ only in $y$. \Cref{fig:y_link_2} visualizes the modification to the link at $y^\phi$.

\begin{figure}[h!]
    \centering
    \includegraphics[width=0.55\textwidth]{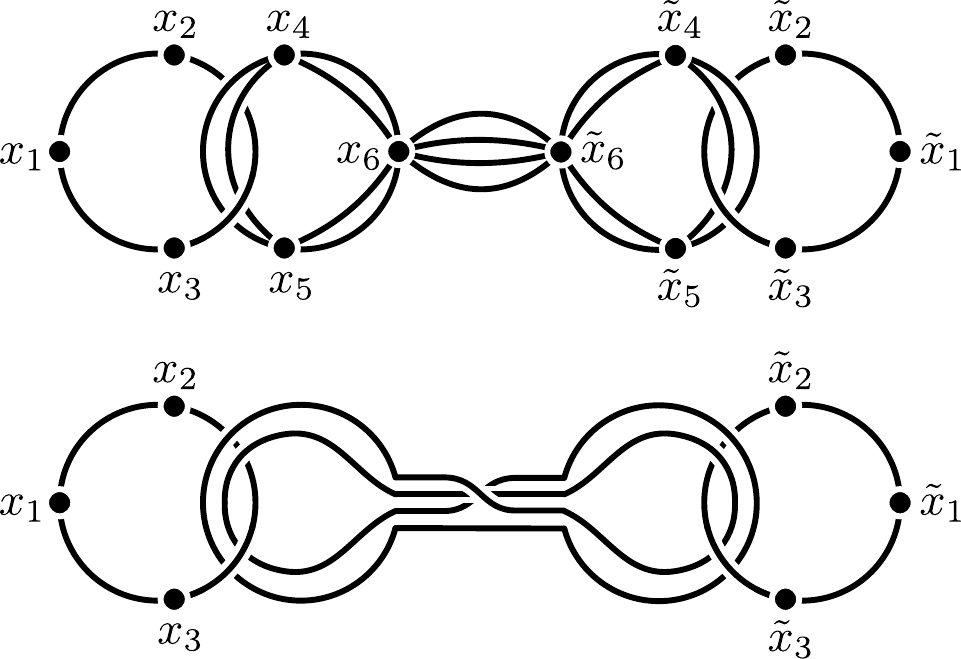}
    \caption{The link of $y^\phi$ after cloning (top) and then joining (bottom) certain 2-cells. The three cycles form Borromean rings.}
    \label{fig:y_link_2}
\end{figure}



It remains to turn $y$ into a vertex of degree three (the center of the $\Y$ in $X_\Y$).
As before, we attach suitable disks $D_1,D_2,D_3\subset\RR^4$ to the complex to modify the link as shown in \cref{fig:y_link_3} (top). 
We then collapse $D_i$ onto edges that we denote as $yz_i$ (see \cref{fig:y_link_3} bottom). The $z_i$ form the neighbors of $y$ in $X_\Y$.
We eventually adjust the subdivisions of the 2-cells so that $y$ is indeed of degree three.
This completes the construction of $X_\Y$ and its embedding.


\begin{figure}[h!]
    \centering
    \includegraphics[width=0.55\textwidth]{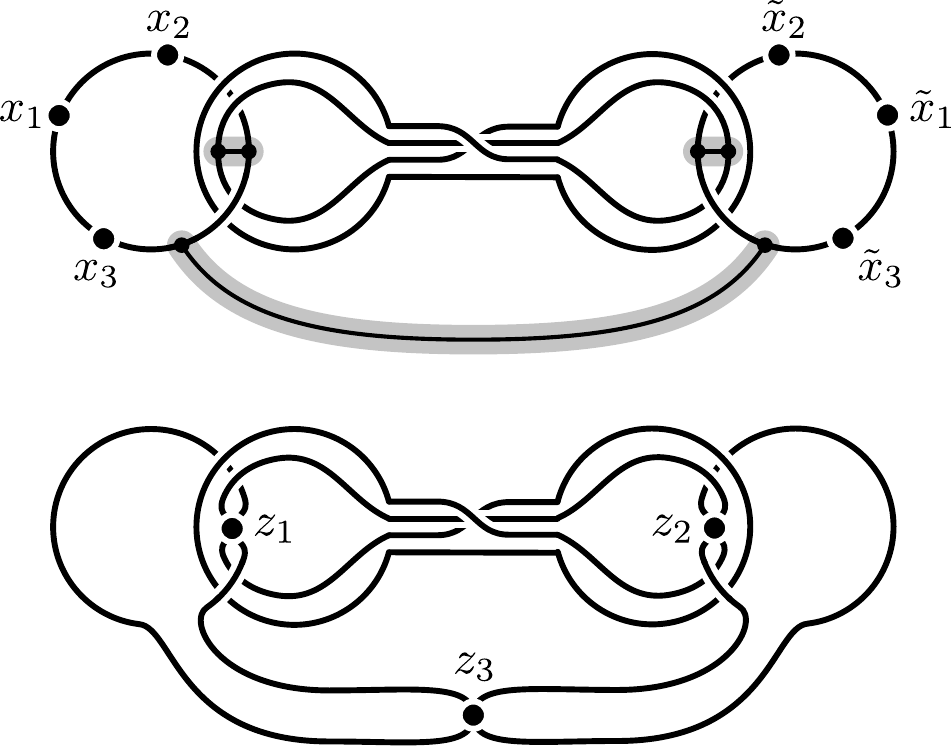}
    \caption{Modification to the link of $y^\phi$ during the last step of the~construction. Link vertices of degree two can be removed by modifying the subdivision of 2-cells adjacent to $y$.}
    \label{fig:y_link_3}
\end{figure}




\tempnewpage

\section{Results on 4-flat graphs}
\label{sec:4_flat}

In \cref{sec:operations} we developed sufficient machinery to already verify van der Holst's \cref{conj:cloning_edges} as well as the $\DY$-part of \cref{conj:DY_YD}. 
In this section we explore further consequences of our results for the theory of 4-flat graphs.

\subsection{Variants of full complexes}
\label{sec:variants_full_complex}

Recall that the \emph{full complex} $\Xany(G)$ of a graph $G$ has 1-skeleton $G$ and a 2-cell attached along each cycle of $G$.
Consider the following variants:
\begin{myenumerate}
    \item $\Xind(G)$ has a 2-cell attached along each \emph{induced} cycle of $G$.
    \item $\Xreg(G)$ has a 2-cell attached along each cycle of $G$ (this is just $\Xany(G)$).
    \item $\Xfull(G)$ has a 2-cell attached along each \emph{closed walk} in $G$.
\end{myenumerate}
The notion of 4-flatness discussed so far is defined by the embeddability of $\Xreg(G)$.
It is natural to consider analogous notions of 4-flatness defined using $\Xind(G)$ and $\Xfull(G)$ respectively. 
We show that they are equivalent:

\begin{theorem}
    \label{res:variants_equal}
    \label{thm:variants_equal}
    The following are equivalent:
    \begin{myenumerate}
        \item $\Xind(G)$ is embeddable.
        \item every regular complex on $G$ is embeddable (or equivalently, $\Xreg(G)$ is). 
        \item \label{it3} every finite (potentially non-regular) complex on $G$ is embeddable (or equivalently, every finite subcomplex of $\Xfull(G)$ is embeddable).\footnote{Note that $\Xfull(G)$ is infinite unless $G$ is just a vertex. Therefore, $\Xfull(G)$ cannot be embeddable itself -- when endowed with the CW topology -- because it is not metrizable, except in degenerate cases. Indeed, $\Xfull(G)$ contains the infinite star as a topological subspace. For this reason we consider only finite subcomplexes of $\Xfull(G)$ here. Alternatively, one can endow $\Xfull(G)$ with a rougher, metrizable, topology, and require that $\Xfull(G)$ be embeddable in item \ref{it3}. The details go beyond the scope of this paper however, and we refer the interested reader to \cite{MOmetCW}.}
    \end{myenumerate}
\end{theorem}
%
\begin{proof}
   The implications \itm3 $\Longrightarrow$ \itm2 $\Longrightarrow$ \itm1 are obvious. Let us first verify \itmto12.

    Every cycle $C\subseteq G$ can be written as a homology sum $C_1\oplus \cdots \oplus C_r$ where each $C_i$ is an induced~cycle of $G$. Indeed, this is easy to prove by induction on the length of $C$, by using a chord $e$ of $C$ to write $C$ as a sum of two shorter cycles containing $e$ as long as $C$ is not induced.  A 2-cell attached along $C$ can be constructed by cloning 2-cells attached along the $C_i$~and joining the clones along shared edges. 
    Since cloning and joining along edges~pre\-serves embeddability (\cf\ Lemmas~\ref{res:joining} and \ref{res:cloning_embeds}), this proves \itmto12.

\medskip
Finally, we verify \itmto23.  A 2-cell with constant attachment map as well as a 2-cell attached along a single edge (traversing it twice in opposite directions) can always be added to a complex while preserving embeddability. 
    Any other walk $\gamma$ in $G$ can be written~as~the~concatenation of cycles and double-traversals of edges; call them $\gamma_1,...,\gamma_r$. 
    A 2-cell attached along $\gamma$ can be constructed by cloning 2-cells attached along the $\gamma_i$ and then joining the clones at the concatenation vertices. Using Lemmas~\ref{res:joining} and \ref{res:cloning_embeds} again, the implication  \itmto23 follows. 
\end{proof}

Since $\Xind(G)$, $\Xreg(G)$ and $\Xfull(G)$ are equivalent from the point of view of embedda\-bility, we will continue to write $\Xany(G)$ to denote any of them and choose the interpretation most convenient for the situation.
For example, $\Xind(G)$ is most convenient to establish embeddability since it has the fewest 2-cells:

\begin{example}
Since all induced cycles of $K_7-e$ are triangles we find $\Xind(K_7-e)$ as a subcomplex of $\mathcal K_7-\Delta$. Since $\mathcal K_7-\Delta$ is embeddable (see \cref{ex:K_7-Delta}), this gives a short proof that $K_7-e$ is indeed 4-flat.    
\end{example}

We present another argument based on Schlegel diagrams of polytopes: 

\begin{example}
    \label{ex:K6_4_flat}
    \label{ex:K7-e_4_flat}
    \label{ex:K6_and_K7-e_4_flat}
    $K_6$ is the 1-skeleton of the 5-dimensional simplex.
    Its 2-dimen\-sional faces form exactly the 2-cells of $\Xind(K_6)$.
    The Schlegel diagram of the simplex is a 3-dimensional complex embedded in $\RR^4$ that contains $\Xind(K_6)$ as a subcomplex.
    We conclude that $K_6$ is 4-flat.
    
    Analogously, consider gluing two 5-dimensional simplices at a 4-dimensional face. 
    The resulting polytope has $K_7-e$ as its 1-skeleton (the missing edge connected the vertices opposite to the glued faces) and each 2-dimensional face corresponds to a 2-cell in $\Xind(K_7-e)$. 
    Following the argument above, we conclude that  $K_7-e$ is 4-flat.
\end{example}





\subsection{Sliced embeddings}
\label{sec:variants_subspace}

We say that an embedding of a 2-complex $X$ is \emph{sliced}, if it embeds the 1-skeleton $G_X$ into the 3-dimensional~subspace $\Pi:=\RR^3\times\{0\}\subset\RR^4$. A number of 4-flat variants can be defined based on~such embeddings. First, we~note the following:


\begin{lemma}
    \label{res:embedd_G_in_subspace}
    Every embeddable 2-complex has a sliced embedding.
\end{lemma}
\begin{proof}
    For a generic embedding $\phi\: X\to\RR^4$ the projection $\pi\:\RR^4\to\Pi$ is injective on the embedded skeleton $\smash{G_X^\phi}$.
    Let $H:=\pi (\smash{G_X^\phi)}\subset\Pi$ be the image of the projection.
    Choose a triangulation $\mathcal T$ of $\Pi$ that contains $H$ as a subcomplex.
    For each simplex $\sigma\in \mathcal T$, let $\hat\sigma:=\pi^{-1}(\sigma)=\sigma\times\RR$ be the cylinder over $\sigma$.
    We~define~a~homeo\-morphism $f\:\RR^4\to\RR^4$ as follows: for a vertex $v\in\mathcal T$, either $v\not\in H$ and $f$ fixes the ray $\hat v$, or $v\in H$ and $f$ translates the ray $\hat v$ by moving the point $\pi^{-1}(v)\cap X^\phi$ to $v$.
    Then extend $f$ linearly to all cylinders $\hat\sigma$. The image $f(X^\phi)$ is a sliced embedding of $X$.
\end{proof}

This suggests to consider the following subtypes of sliced embeddings:
%
%
\begin{myenumerate}
    \item a \emph{$+$-sliced embedding} embeds each 2-cell into the halfspace $\RR^3\times\RR_+$.
    \item a \emph{$\pm$-sliced embedding} embeds each 2-cell into one of the halfspaces $\RR^3\times\RR_+$ and $\RR^3\times\RR_-$. 
\end{myenumerate}
%

It turns out that graphs for which $\Xany(G)$ has a $+$-sliced~embedding are exactly the linkless graphs: recall that a graph $G$ is \emph{linkless} if there~is an embedding $\phi\: G\to\RR^3$ so that any two disjoint cycles of $G$ are mapped to unlinked closed curves.

\begin{theorem}
    \label{res:Rplus_iff_linkless}
    The following are equivalent for every graph $G$:
    \begin{myenumerate}
        \item $\Xany(G)$ has a $+$-sliced embedding.
        \item $G$ is linkless.
    \end{myenumerate}
\end{theorem}
\begin{proof}
    We start with the proof of \itm2 $\Longrightarrow$ \itm1, which is essentially due to van der~Holst \cite[Theorem 2]{van2006graphs}, but we include it here for completeness and for its elegance.
    If~$G$~is linkless, then it is known that it has a \emph{flat} embedding $\phi\:G\to \RR^3$ \cite{RoSeThoSac}, that is, for each 2-cell $c\subseteq \Xreg(G)$ ---which is attached along a cycle in $G$--- there exists an embedded disk $D_c\:c\to \RR^3$ with $D_c\cap \embd G=\partial D_c \cap G^\phi=\partial c$.
    We choose distinct numbers~$a_c>0$, one for each 2-cell $c\subseteq\Xreg(G)$.
    We now extend $\phi$ to embed the 2-cell $c$ into $\RR^3\times$ $\RR_+$ by the following map:
    $$c\ni x \;\mapsto\; \phi(x):=\begin{pmatrix}
        D_c(x)
        \\[0.3ex]
        a_c\dist(D_c(x),\embd G)
    \end{pmatrix} \in\RR^3\times\RR_+
    $$
    %
    %
    where $\dist(D_c(x),\embd G)$ is the distance of $D_c(x)$ to the closest point in $G^\phi$.
    It remains to show that this is an embedding.
    Clearly, $\phi$ is an embedding on $G$ with any single 2-cell.
    If two embedded 2-cells $c_1^\phi$ and $c_2^\phi$~were~to~inter\-sect in points $\phi(x_1)=\phi(x_2)$, where $x_i\in c_i$, then comparing the $\phi(x_i)$ component-wise yields
    $$a_{c_1}\dist(D_{c_1}(x_1),\embd G)=a_{c_2}\dist(D_{c_2}(x_2),\embd G)=a_{c_2}\dist(D_{c_1}(x_1),\embd G),$$
    and hence $a_{c_1}=a_{c_2}$, implying  $c_1=c_2$, which is a contradiction.
    

    For the implication \itm1 $\Longrightarrow$ \itm2, note first that if $G$ is not linkless, then each embedding in $\RR^3$ contains two disjoint cycles of non-zero linking number (by definition it contains two linked cycles, but the non-zero linking number follows from the excluded-minor characterization of linkless graphs and the fact that such a pair of cycles exists for each excluded minor; see \cite{sachs2006spatial} or \cite[Theorem 36]{van2009graph}). 
    Secondly,\nls a pair of disjoint closed curves in $\RR^3\times\{0\}$ of non-zero linking number cannot be filled in by disjoint disks in $\RR^3\times\RR_+$ (see \eg\ \cite[Lemma 2]{skopenkov2021short}). These two facts combined prove \itm1 $\Longrightarrow$ \itm2.
\end{proof}

In contrast, graphs $G$ for which $X(G)$ has a $\pm$-sliced embeddings seem to be a much larger class. While we suspect that not every~embed\-dable~2-complex (or full complex of a 4-flat graph) has a $\pm$-sliced embedding, we~do~not know of any examples.

\begin{question}
    \label{q:pm_sliced}
    \quad
    \begin{myenumerate}
        \item Does every embeddable 2-complex have a $\pm$-sliced embedding?
        \item If $G$ is 4-flat, must its full complex $\Xany(G)$ have a $\pm$-sliced embedding?
    \end{myenumerate}
\end{question}



\subsection{4-flat and linkless graphs}
\label{sec:4flat_linkless}
\label{sec:4_flat_linkless}

We say that a graph $G$ is \emph{locally linkless}, if the~neighborhood $N_G(x)$ of each vertex $x\in G$ is a linkless graph.

\begin{lemma}
    \label{res:locally_linkless}
    4-flat graphs are locally linkless.
\end{lemma}
\begin{proof}
    Fix an embedding $\phi\: \Xany(G)\to\RR^4$.
    For a vertex $x$ of $G$, let $B\subset\RR^4$ be a link neighborhood of $x^\phi$.
    Note that $\Xany(G)$ contains a cone $C$ over $N_G(x)$ with apex at $x$.
    Then the closure of $C^\phi\setminus B$ is a cylinder of the form $N_G(x)\times[0,1]$.
    Collapsing this cylinder onto $C\cap \partial B$ preserves embeddability by \cref{res:collapsing_cylinder}.
    This results in an embedding of $\Xany(G)$ for which $N_G(x)$ is embedded into $\partial B$, and each 2-cell non-incident with $x$ is embedded outside of $B$.
    In particular, it yields a $+$-sliced embedding of $\Xany(N_G(x))$.
    Thus $N_G(x)$ is linkless by \cref{res:Rplus_iff_linkless}.
\end{proof}

Using \cref{res:locally_linkless} we obtain a short argument for the following well-known fact:

\begin{corollary}
    \label{res:K7_K3311_short}
    $K_7$ and $K_{3,3,1,1}$ are not 4-flat.
\end{corollary}
\begin{proof}
    The neighborhood of a vertex in $K_7$ is $K_6$.
    The neighborhood of a domina\-ting vertex in $K_{3,3,1,1}$ is $K_{3,3,1}$. 
    But the graphs $K_6$ and $K_{3,3,1}$ are not linkless~\cite{RoSeThoSac}. 
\end{proof}

\subsection{4-flatness and suspensions}
\label{sec:4flat_coning}

Given graphs $G$ and $H$, we let $G*H$  denote the graph join of $G$ and $H$, \ie\, the graph obtained from the disjoint union $G\cupdot H$ by adding the edges of a complete bipartite graph between the vertices of $G$ and the vertices of $H$.
For example, $G*K_1$ is a single suspension of $G$, and $G*K_n$ is an $n$-fold (iterated) suspension.

\begin{theorem} 
\label{res:linkless_planar_outerplanar}
The following hold for every graph $G$:
\begin{myenumerate}
    \item $G*K_1$ is 4-flat if and only if $G$ is linkless.
    \item $G*K_2$ is 4-flat if and only if $G$ is planar.
    \item $G*K_3$ is 4-flat if and only if $G$ is outerplanar.
\end{myenumerate}
\end{theorem}
\begin{proof}
    A graph is outerplanar if and only if its suspension is planar; and a graph is planar if and only if its suspension is linkless \cite{van1999colin}.
    It therefore suffices to prove \itm1.

    Let $x$ be the suspension vertex of $G*K_1$.
    If $G*K_1$ is 4-flat, then linklessness of $G$ follows from \cref{res:locally_linkless}  and the fact that $G\simeq N_{G*K_1}(x)$.
    For the converse, observe that $\Xind(G*K_1)$ is obtained from $\Xind(G)$ by adding a cone over $G$.
    If $G$ is linkless, then by \cref{res:Rplus_iff_linkless} there exists a $+$-sliced embedding $\phi\: \Xind(G)\to\RR^3\times\RR_+$. 
    We can extend $\phi$ to an embedding of $\Xind(G*K_1)$ by embedding the cone over $G$ into $\RR^3\times\RR_-$.
    Thus $G*K_1$ is 4-flat.
\end{proof}


The join operation can be used to define potentially interesting intermediate~classes between outerplanar, planar and linkless graphs. Let $\overline G$ denote the complement of $G$:

\begin{example}
    \label{ex:intermediate_coning}
    $G*\overline K_2$ 
     is 4-flat for some linkless graphs (\eg\ $G=K_5$ gives $G*\overline K_2$ $=K_7-e$), but not for others (\eg\ $G=K_{3,1,1,1}$ gives $G*\overline K_2=K_{3,3,1,1}$).
    Likewise, $G*\overline K_3$ is 4-flat for some planar graphs (\eg\ $G=K_4$ gives $G*\overline K_3$ $=K_7-\Delta$), but not for others (\eg\ $G=K_{3,1,1}$ gives $G*\overline K_3=K_{3,3,1,1}$).
\end{example}

\begin{problem}
    \label{prob:coning}
    \label{q:coning}
    Given a graph $H$, describe the class of graphs $G$ for which $G*H$ is 4-flat. 
\end{problem}


\subsection{4-flat and knotless graphs}
\label{sec:4flat_knotless}

A graph $G$ is \emph{knotless}, if it has an embedding $\phi\: G\to\RR^3$ for which each cycle $C\subseteq G$ is embedded as the trivial knot.
A graph that has no such embedding is called \emph{intrinsically knotted}.
The knotless graphs are often considered as another natural continuation from planar and linkless graphs,  
and one could ask about their relation to 4-flat graphs.
The following two examples demonstrate that these two classes are generally unrelated:



\begin{example}
    \label{ex:knotless_notimpl_4flat}
    While the \DY-family of $K_{3,3,1,1}$ consists entirely of intrinsically knotted graphs, seven out of the 20 members of the \DY-family of $K_7$ are in fact linkless \cite[Table 1]{goldberg2014many}.
    Yet, being Heawood graphs, none of them is 4-flat.
    Thus, there are graphs that are knotless but not  4-flat.
\end{example}

\begin{example}
    \label{ex:4flat_notimpl_knotless_sliced}
    There exists a linkless graph $G$ whose suspension $G*K_1$ is not knotless \cite{foisy2003newly}.
    However, by \cref{res:linkless_planar_outerplanar} \itm1, $G*K_1$ is 4-flat.
    Thus there are graphs that are 4-flat but not knotless.
\end{example}

\subsection{Clique sums of 4-flat graphs}
\label{sec:clique_sums}

Given two graphs $G_1$ and $G_2$ both of which contain a $k$-clique $K_k$, a \emph{$k$-clique sum} $G_1 *_k G_2$ is a graph obtained by ``gluing'' $G_1$ and $G_2$ by identifying a $K_k$ subgraph of $G_1$ with one of $G_2$.

\begin{proposition}
    \label{res:3_clique_sums}
    The class of 4-flat graphs is closed under 3-clique sums.
\end{proposition}
\begin{proof}
    Let $G_1$ and $G_2$ be two 4-flat graphs with 3-cliques $\Delta_i\subseteq G_i$ and embeddings $\phi_i:X(G_i)\to \Sph^4$.
    Let $c_i\subseteq X(G_i)$ be the 2-cell with $\partial c_i=\Delta_i$ and~$x_i\in\Int(c_i)$ an interior point at which $\phi$ is \emph{locally flat},
    that is, given~a~link~neighbor\-hood $B_i\subset\RR^4$ of $x_i^{\phi_i}$, the link $\partial B_i\cap c_i^{\phi_i}$ is unknotted in $\partial B_i$.
    There is then~a~homeomorphism $(\partial B_1,\partial B_1\cap c_1^{\phi_1})\simeq (\partial B_2,\partial B_2\cap c_2^{\phi_2})$.
    We identify $\Sph^4\setminus B_1$ and $\Sph^4\setminus B_2$ and their embedded complexes $(\Xany(G_i),\phi_i)$ along this homeomor\-phism, which yields a new $\Sph^4$ with an embedded complex $X'$.
    This new complex $X'$ contains a cylinder with boundary $\Delta_1\cup \Delta_2$.
    We collapse this cylinder (thereby identifying $\Delta_1$ and~$\Delta_2$), which preserves embeddability by \cref{res:collapsing_cylinder}.
    The resulting complex $X''$ has~1-skeleton $G_1 *_3 G_2$.
    Since each induced cycle of $G_1 *_3 G_2$ is contained in either $G_1$ or $G_2$, we find $\Xind(G_1 *_3 G_2)$ as a sub-comples of $X''$. 
    Thus $G_1 *_3 G_2$ is  4-flat by \cref{res:variants_equal}.
\end{proof}

It follows easily that $k$-clique sums for $k<3$ also preserve 4-flatness. For example, each 2-clique can be turned into a 3-clique by cloning and subdividing edges, which preserves 4-flatness by \cref{res:cloning_edges_4_flat}. 


\begin{question}
    \label{q:4_clique_sum}
    Do 4-clique sums preserve 4-flatness?
\end{question}

This question generalizes \cref{q:stellifying_K4}: observe that adding a suspension over a complete subgraph $K_4$ is the same as a 4-clique sum with $K_{5}$. 



Note that 5-clique sums do \emph{not} preserve 4-flatness: starting from three disjoint copies of $K_6$ (which is 4-flat) we can use 5-clique sums to glue them at $K_5$-subgraphs to obtain $K_5*\bar K_3$. 
The full complex of this graph~contains the triple cone over $K_5$, which is not embeddable.

\tempnewpage

\section{Heawood graphs are excluded minors}
\label{sec:Heawood_excluded_minor}

Van der Holst conjectured that the Heawood graphs are exactly the excluded~minors for the class of 4-flat graphs (\cref{conj:Heawood}).
Comparing with the analogous result of Robertson,  Seymour \& Thomas for  linkless graphs \cite{RoSeThoSac}, proving \cref{conj:Heawood}\ seems very hard and out of reach of our techniques.

The first step towards this conjecture is to confirm that each of the 78 Heawood graphs is indeed an excluded~minor, and we do so with the following theorem. The difficulty in proving this is not only due to the large number of graphs that need to be checked; more importantly, as mentioned in the introduction, there is no known algorithm for checking whether a given graph is 4-flat. 


\begin{theorem}
\label{res:all_Heawood_are_excluded}
    Each graph of the Heawood family is an excluded minor for the class of 4-flat graphs.
\end{theorem}

We will offer two proofs. 
The first proof relies on computer help, enumerating all minors of all Heawood graphs and checking them one by one (\cref{sec:proof_by_computer}). 
Our second proof first uses results from \cref{sec:4_flat} to reduce the case analysis to merely five Heawood graphs. 

The general procedure for both proofs is as follows:
given a Heawood graph $G$, we~step through its minors $H$ (whereby it suffices to consider minors of the form $G-e$ and $G/e$) and show that each one is 4-flat.
We do so by finding two vertices~$v,w$ $\in V(H)$ for which $H':=H-\{v,w\}$ is planar. 
This shows that $H$ is a subgraph of $H'*K_2$ and therefore 4-flat by \cref{res:linkless_planar_outerplanar}~\itm2.

Note that this approach for detecting 4-flat graphs was not guaranteed to \mbox{succeed} a priori: 
there are 4-flat graphs that are not contained in the double suspension of a planar graph (\eg\ the disjoint union of two $K_6$).
Computationally we found that this works at least for the minors of Heawood graphs.
The following question remains:

\begin{question}
    \label{q:are_all_4_flat_suspension_of_planar}
    If $G$ is an excluded minor for the class of 4-flat graphs and $H$ is a minor of $G$, then are there two vertices $v,w\in V(H)$ so that $H-\{v,w\}$ is planar?

\end{question}

\subsection{Proof by computer}
\label{sec:proof_by_computer}

Our code for both enumerating all Heawood graphs and~performing the exhaustive case analysis can be found in \cref{sec:appenix_code}. 
The output of~the program confirms \cref{res:all_Heawood_are_excluded}.

\newcommand{\forb}[1]{\mathrm{Forb}(#1)}
\newcommand{\ex}[1]{\mathrm{Ex}(#1)}
\newcommand{\ff}{\ensuremath{\rm{F}^4}}


\subsection{Proof by hand}
\label{sec:proof_by_hand}

In \cref{sec:operations} we proved that 4-flat graphs~are closed under cloning edges (\cref{res:cloning_edges_4_flat}) and \DYtrafos\ (\cref{res:DY_YD_4_flat}).
This constitutes an affirmative answer to van der Holst's conjecture \cref{conj:cloning_edges} and part of \cref{conj:DY_YD}.
Assuming the truth of both conjectures, van der Holst~showed that all Heawood graphs are indeed excluded minors for the class of 4-flat graphs \cite[Lemma 1]{van2006graphs}. 
However, since we only verified the \DY-part of \cref{vdH_DY_YD} we cannot draw this conclusion yet.
\mbox{Our results~do} however sufficiently reduce the work that remains to be done, so as to allow for a proof by hand.

The following lemma is used to reduce the case analysis to only a handfull of graphs:

\begin{lemma}
\label{cor Ex DY}
\label{res:DY_preserves_excluded}
    Let $G$ be an excluded minor for the class of 4-flat graphs and let~$H$~be a graph obtained from $G$ by a \DYtrafo.
    Then either $H$ is 4-flat, or~$H$~is~itself an excluded minor for the class of 4-flat graphs.
\end{lemma}

\begin{proof}
    It suffices to show that either $H$ is 4-flat, or every proper minor $H'$ of $H$ is 4-flat.
    Since the class of 4-flat graphs is minor-closed, it suffices to consider the cases $H'=H-e$ and $H'=H/e$ for some edge $e\in E(H)$. 
    
    Let $x,y,z\in V(H)$ be the vertices of the triangle $\Delta\subset G$ on which we performed the \DYtrafo, and let $v$ by the resulting vertex of degree three in $H$.\nls
    There are two cases to be considered.

    \textit{Case 1:} $e$ does \textit{not} contain $v$.
    Then $e$ is also an edge of $G$, and $\Delta$ is also a triangle in $G':=G-e$ \resp\ $G':=G/e$, and $H'$ is obtained from $G'$ by a \DYtrafo\ on $\Delta$.
    Since $G$ is an excluded minor, $G'$ is 4-flat, and so $H'$ is 4-flat too by \cref{res:DY_YD_4_flat}.
    
    \textit{Case 2:} $e$ does contain $v$.
    %
    We assume without loss of generality that $e=vx$. 
    Then $H/e$ coincides with $G - yz$, which is 4-flat. 
    Similarly, $H-e$ is obtained from $G-\{xy,xz\}$~by~subdividing the edge $yz$, which again is 4-flat.  
\end{proof}



By \cref{res:DY_preserves_excluded}, to prove \cref{res:all_Heawood_are_excluded} it suffices to consider the Heawood graphs that are not \DYtrafos\ of other Heawood graphs. 
%
If those are shown to be excluded minors for the class of 4-flat graphs, then it follows that all Heawood graphs are excluded minors (using van der Holst's result that all Heawood graphs are not 4-flat).

Note that these are exactly the Heawood graphs of minimum degree $\ge 4$, because we cannot perform a \YDtrafo\ on them.
Computer code for listing these Heawood graphs can be found in \cref{sec:appenix_remaining_Heawood}.
This list consists of $K_7$ and $K_{3,3,1,1}$ (which are known to be excluded minors; \cite[Lemma 2]{van2006graphs}) and five other graphs shown in \cref{fig:remaining_Heawood}.
We call them the \emph{remaining Heawood graphs}.


\begin{figure}[h!]
    \centering
    \includegraphics[width=0.95\textwidth]{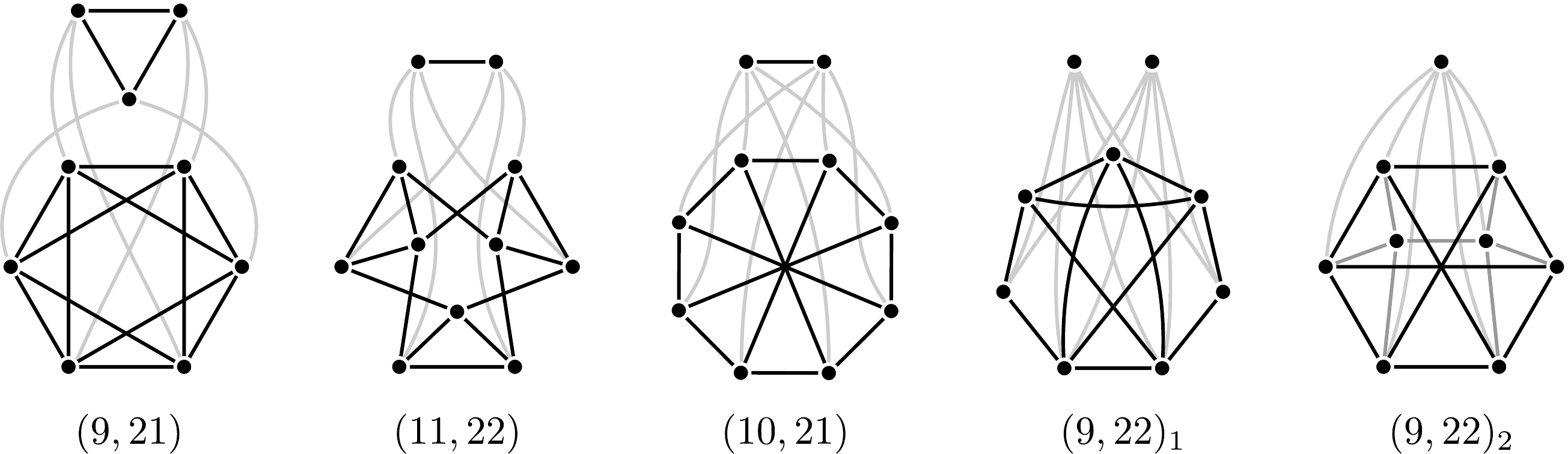}
    \caption{The five remaining Heawood graphs. 
    The edges are colored solely for visualization. The name $(v,e)$ indicates that the graph has $v$ vertices and $e$ edges, with a sub-index to resolve ambiguities.}
    \label{fig:remaining_Heawood}
\end{figure}


In the following subsections we treat the remaining Heawood graphs one by one. For each of these graphs $H$, and each $e\in E(H)$, we will prove that both $H-e$ and $H/e$ can be made planar by deleting two vertices; as mentioned above, this implies that $H-e$ and $H/e$ are 4-flat. Our graphs have many automorphisms, and this will help us reduce the number of cases by considering only one edge $e$ from each orbit of the automorphism group of $H$.

\subsubsection{The case $(9,21)$}

There are at most four
orbits of edges under the automorphisms of the graph, as colored in \cref{FigHexagon}~(a). The bottom half is the square of a hexagon, which has a planar embedding $D$ as shown in \cref{FigHexagon}~(b). 


\begin{figure}[h!]
\begin{center}
\begin{overpic}[width=.7\linewidth]{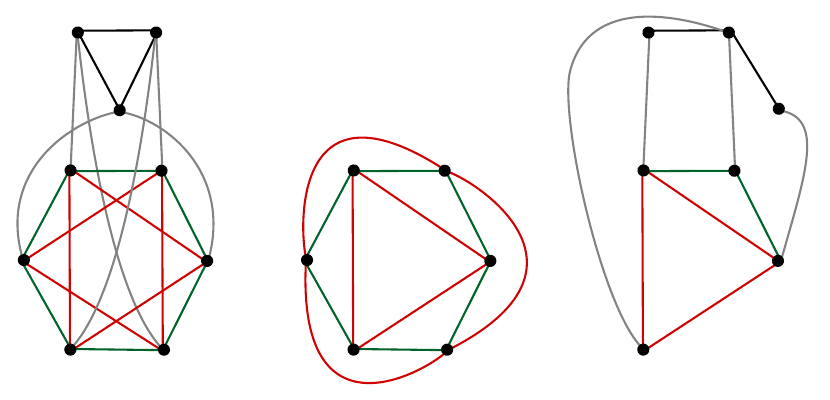} 
\put(7,47){$u$}
\put(20,46){$v$}
\put(88,47){$v$}
\put(76,47){$u$}
\put(11,-2){$(a)$}
\put(46,-2){$(b)$}
\put(82,-2){$(c)$}
\end{overpic}
\end{center}
\caption{Proving that all proper minors of $(9,21)$ are 4-flat.} \label{FigHexagon}
\end{figure}
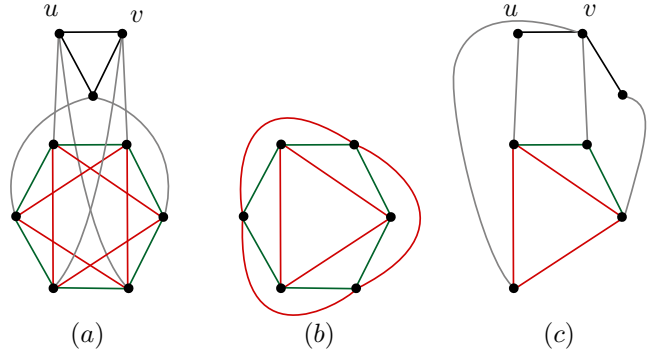

If a red or green edge is deleted, then two of the faces are joined into a 4-gonal face $F$, which has two opposite vertices $x,y$ in its boundary. Embed the top vertex $u$ that sends edges to  $x,y$ inside $F$, and delete the two remaining top vertices.

Similarly, if a grey edge $ux$ is deleted, we embed $u$ with its one remaining grey edge inside a face of $D$, and delete the two remaining top vertices.

If a black edge $uv$ is deleted, remove 
two bottom vertices to obtain a graph embeddable as in \cref{FigHexagon}~(c).

If a grey edge $ux$ is contracted, we are in an easier situation than the previous one: we can delete the resulting vertex along with one more bottom vertex to obtain a subgraph of the previous case.

If a black edge $uv$ is contracted, delete it along with the remaining top vertex to obtain $D$.

Finally, if a red or green edge $xy$ is contracted, delete it along with a bottom vertex $z$ such that $x,y,z$ are consecutive along the green cycle. We are left with the top triangle, a triangle of $D$, and a perfect matching joining these triangles. This graph is a triangular prism, hence planar.

\subsubsection{The case $(11,22)$}

There are five orbits of edges as colored in \cref{FigCrossings}~(a). Let $u,v$ denote the top vertices. 

\begin{figure}[h!]
\begin{center}
\begin{overpic}[width=.7\linewidth]{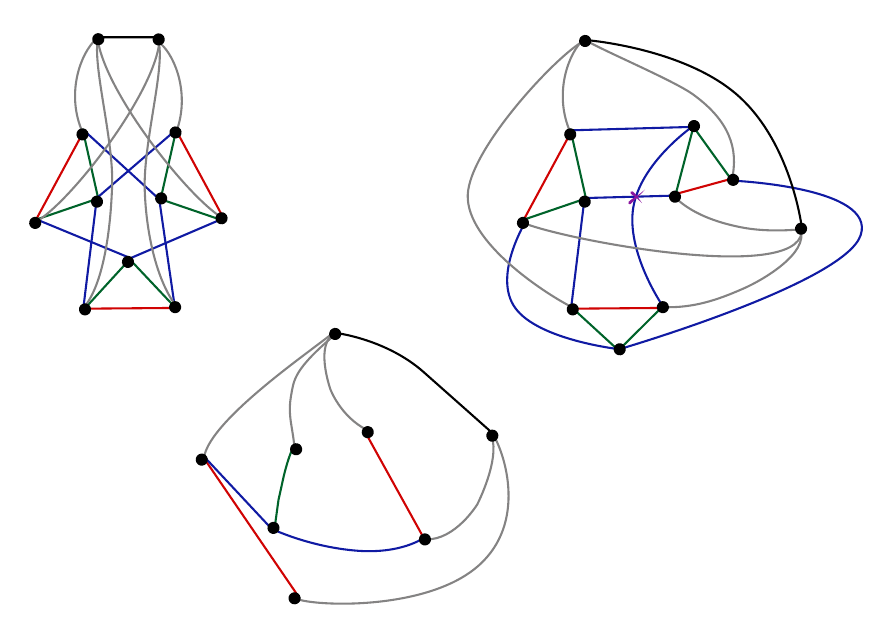} 
\put(8,67){$u$}
\put(20,67){$v$}
\put(5.5,55){$w$}
\put(1,45){$z$}
\put(21,36){$x$}
\put(26.5,45.5){$d$}
\put(15.5,48){$y$}
\put(6,36){$a$}
\put(12,46.5){$b$}
\put(21,56){$c$}
\put(14.5,42.2){$t$}

\put(93,44.5){$v$}
\put(64,67){$u$}
\put(61,55.5){$w$}
\put(57,45){$z$}
\put(77,36){$x$}
\put(85.5,50.5){$d$}
\put(78,58){$y$}
\put(62.5,35){$a$}
\put(68,49){$b$}
\put(76,45.5){$c$}
\put(71,28){$t$}

\put(38,34){$u$}
\put(58,21){$v$}
\put(49,10.5){$x$}
\put(30.5,1){$z$}
\put(43,21){$a$}
\put(32,11){$y$}
\put(30,20){$d$}
\put(20,19){$w$}

\put(10,29){$(a)$}
\put(80,29){$(b)$}
\put(45,-3){$(c)$}
\end{overpic}
\end{center}
\caption{Proving that all proper minors of $(11, 22)$ are 4-flat.} \label{FigCrossings}
\end{figure}
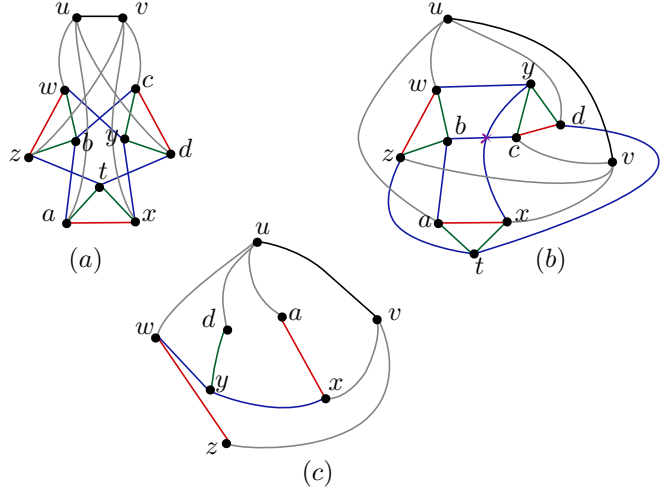

A different drawing, with five crossings of edges, is shown in \cref{FigCrossings}~(b). If the blue $xy$ or grey $zv$ edge is deleted, we remove vertices $t,b$ leaving no crossings. If the black edge $uv$ is deleted, we remove vertices $a,x$. Notice that if vertices $u,v$ are deleted, then only one crossing remains, namely the one marked with a purple $\times$ symbol. Moreover, if one of the red $cd$ or green $yd$ edges are deleted, then we can reroute the edge $xy$ to avoid any crossing. Thus if a red or green (or blue) edge is deleted, we can delete $u,v$ to obtain a planar graph.

If a red or green edge $e$ is contracted, we remove the resulting vertex as well as the third vertex forming a red-green triangle with $e$. Since all crossings of \cref{FigCrossings}~(b) involve edges incident with the bottom triangle, this results in a planar graph. If the black $uv$ or grey $vz$ edge is contracted, we remove the resulting vertex as well as the remaining vertex among $u,v,z$. Then only the crossing marked with a $\times$ symbol remains, and we can re-route the $xy$ edge to avoid it. Finally, if the  blue $bc$ edge is contracted, we remove 
    the resulting vertex along with $t$, and draw the resulting graph as in \cref{FigCrossings}~(c).                  

\subsubsection{The case $(10,21)$}

There are four orbits of edges as colored in \cref{FigOctagon}~(a). The bottom part has a drawing with only one crossing as shown in \cref{FigOctagon}~(b). 

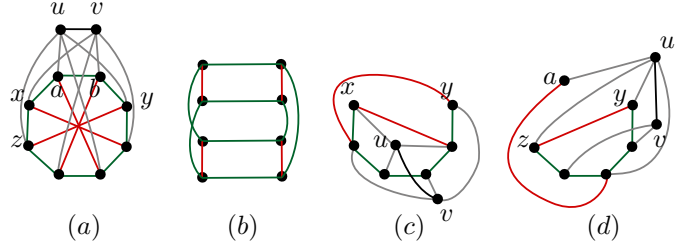
\begin{figure}[h!] 
\begin{center}
\begin{overpic}[width=.7\linewidth]{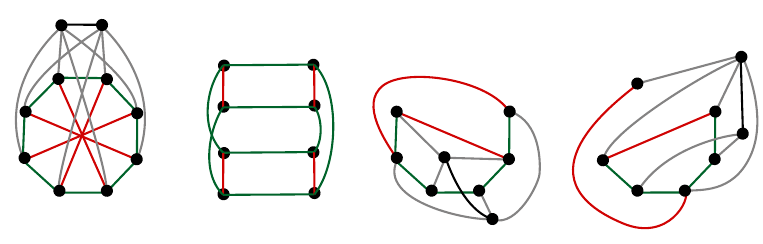} 
\put(6.5,31){$u$}
\put(12.5,31){$v$}
\put(6.5,18.5){$a$}
\put(12.3,18.5){$b$}
\put(0.5,18){$x$}
\put(20,17.5){$y$}
\put(0.5,11){$z$}

\put(98,26){$u$}
\put(97,11){$v$}
\put(65,19){$y$}
\put(50,19){$x$}

\put(80.5,20.5){$a$}
\put(55,11){$u$}
\put(65,0){$v$}
\put(91,18){$y$}
\put(76.5,11){$z$}

\put(9,-2){$(a)$}
\put(33,-2){$(b)$}
\put(58,-2){$(c)$}
\put(87,-2){$(d)$}
\end{overpic}
\end{center}
\caption{Proving that all proper minors of $(10,21)$ are 4-flat.} \label{FigOctagon}
\end{figure}

If we delete a green edge, then the aforementioned crossing disappears, and so we obtain a planar subgraph after removing $u,v$.

Similarly, if a black or grey edge is  contracted, then by removing it along with one more vertex we can obtain the subgraph of \cref{FigOctagon}~(b) where one of the vertices involved in the crossing is removed. 

Consider now the subgraph obtained by removing $a,b$, as drawn in \cref{FigOctagon}~(c). There is just one crossing, between the black edge and a green edge. Thus if the green edge $ab$ is contracted, we can delete it along with one of the vertices involved in said crossing to obtain a planar subgraph. If the black edge is deleted, then we remove $a,b$. If a red edge is contracted, assume it is one incident with $a$, and remove it along with $b$; again the crossing of \cref{FigOctagon}~(b) thereby disappears.

For the remaining two cases, remove $x,b$ and consider the drawing of the resulting subgraph in \cref{FigOctagon}~(d). If a grey edge is deleted, the unique crossing disappears. Finally, if a red  edge is deleted, assume it is $zy$, and embed $v$ at a midpoint of where $zy$ used to lie.

\subsubsection{The case $(9,22)_1$}

    Let $u,v$ denote the top vertices, and note that $G-\{u,v\}$ is a subdivision of $K_5$.

\begin{figure}[h!] 
\begin{center}
\begin{overpic}[width=.7\linewidth]{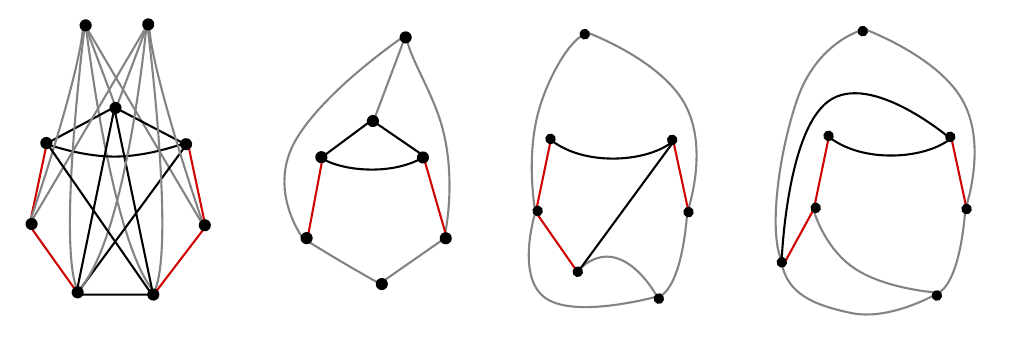} 
\put(6,32){$u$}
\put(15.5,32){$v$}
\put(11,24){$x_0$}
\put(-1,11){$x_1$}
\put(3,5){$x_2$}

\put(41,31){$v$}
\put(37,3){$u$}
\put(34,23.5){$x_0$}
\put(27,7){$x_1$}

\put(57,31){$u$}
\put(66,3){$v$}
\put(48.5,12){$x_1$}
\put(53,5){$x_2$}

\put(82,32){$u$}
\put(93,3){$v$}
\put(81,12){$x_1$}
\put(74,5){$x_2$}

\put(9.5,-1){$(a)$}
\put(35,-1){$(b)$}
\put(58,-1){$(c)$}
\put(84,-1){$(d)$}
\end{overpic}
\end{center}
\caption{Proving that all proper minors of $(9, 22)_1$ are 4-flat.} \label{FigK5}
\end{figure}
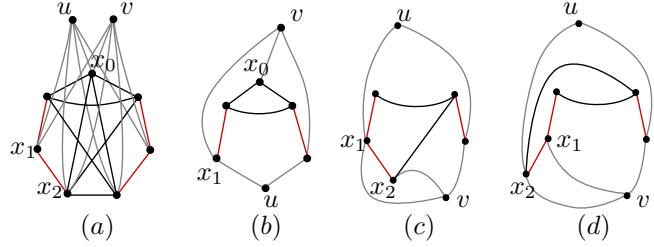        
    If one of the four red edges shown in \cref{FigK5}~(a) is contracted, then by deleting the contracted vertex as well as its red neighbour we are left with a subgraph of the double suspension of a 4-cycle, which is planar. 

    If an edge $ux$ of $u$ is deleted, we consider two subcases: a) if $x$ is the top vertex $x_0$, we delete the two bottom vertices. The result is a planar graph shown in \cref{FigK5}~(b). \\ 
    b) if not, then we may assume that $x$ is one of $x_1,x_2$ because of the symmetry. In this case we delete $x_0$, as well as $z$. The result is a planar graph shown in \cref{FigK5}~(c) or (d), depending on which of $x_1,x_2$ is $x$.

    If an edge $ux$ of $u$ is contracted, then we remove $v$ and the contracted vertex. We are left with a proper minor of $K_5$, hence a planar graph.

    By symmetry, if an edge of $v$
is deleted or contracted we are happy.

    In all other cases, we remove $u,v$. We are left with a subdivision of $K_5$, where the red edges are the ones arising from a subdivision. Thus if we delete any edge, or contract an edge that is not red, we obtain a planar graph.

\subsubsection{The case $(9,22)_2$}

There are seven orbits of edges as colored in \cref{FigCube}~(a). We embed $G - u$ with four crossings as in \cref{FigCube}~(b). If we delete an edge $e$ not incident with $u$, then we remove $u$ as well as one more vertex to obtain a planar subgraph as follows:
\begin{itemize}
    \item if $e=ay$ (red), we remove $c$;
    \item if $e=xb$ (green), we remove $c$;
    \item if $e=xc$ (black), we remove $y$;
    \item if $e=cy$ (pink), we remove $x$;
    \item if $e=cd$ (blue), we remove $x$.
\end{itemize}

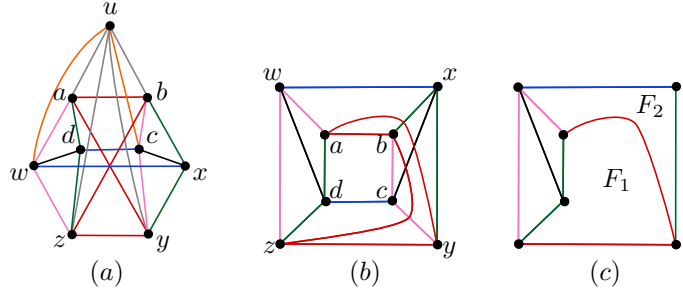
\begin{figure}[h!] 
\begin{center}
\begin{overpic}[width=.7\linewidth]{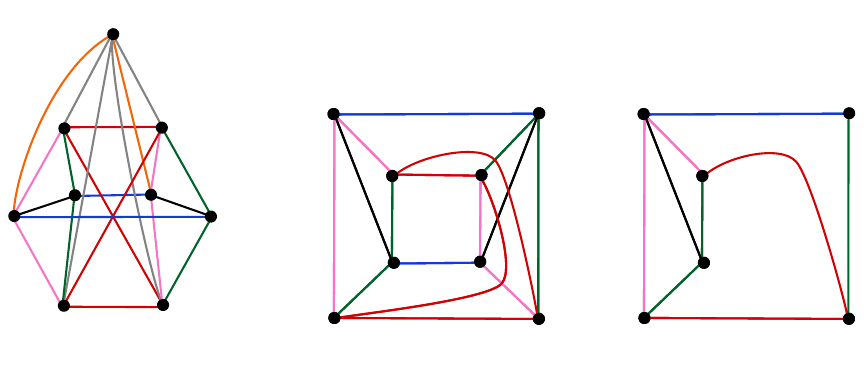} 
\put(12,40){$u$}
\put(4.5,27){$a$}
\put(20,27){$b$}
\put(18.5,20.5){$c$}
\put(5.8,20.5){$d$}
\put(-2,15){$w$}
\put(25.5,15){$x$}
\put(4.5,5){$z$}
\put(20,5){$y$}

\put(46,19){$a$}
\put(53,19){$b$}
\put(53,12.3){$c$}
\put(46,12.3){$d$}
\put(36,30){$w$}
\put(63,30){$x$}
\put(36,4){$z$}
\put(63,4){$y$}

\put(87,14){$F_1$}
\put(92,25){$F_2$}

\put(10,0){$(a)$}
\put(49,0){$(b)$}
\put(85,0){$(c)$}
\end{overpic}
\end{center}
\caption{Proving that all proper minors of $(9,21)_2$ are 4-flat.} \label{FigCube}
\end{figure}

If we delete an edge incident with $u$, say $e=uw$ (respectively, $e=uz$), then we remove vertices $c,b$, and embed the remaining graph as in \cref{FigCube}~(c), putting $u$ inside face $F_1$ (resp.\ $F_2$).

If we contract an edge $f$, then we remove the contracted vertex and proceed as follows:  
\begin{itemize}
    \item if $f=ab$ (red), we can remove $u$ as \cref{FigCube}~(b) becomes planar; 
    \item if $f=xy$ (green), we remove $u$ and obtain a subgraph of the above case $e=xc$;
    \item if $f=cb$ (pink), or $f=xc$ (black), we remove $u$ and obtain a subgraph of the above case $e=xb$;
    \item if $f=cd$ (blue), we remove $x$, and notice that the remaining subgraph is induced by $u$ and a tree, and is therefore planar;
    \item if $f=uy$ (grey), or $f=uc$ (orange), we remove $b$;
\end{itemize}


\tempnewpage

\section{Further open questions}
\label{sec:conlcusions}
\label{sec:outlook}
\label{sec:open_questions}


\iftrue 


The study of 4-flatness comes with a number of difficulties and subtleties inherited from the embedding problem $2\to 4$.
The standard homological obstruction, the van Kampen obstruction \cite{van1933komplexe}, is not strong enough to reliably decide embeddability of 2-complexes in $\RR^4$, yet already the simplest example for this failure (the Freedman-Krushkal-Teichner complex \FKT, see \cref{sec:YD_counterexample}) is rather complex.
Yet, for all known graphs $G$ that are not 4-flat, their full complex $X(G)$ has non-zero van Kampen obstruction.

\begin{question}\label{q:van_Kampen}
    Is it the case that a graph $G$ is 4-flat if and only if its full complex $X(G)$ has zero van Kampen obstruction?
\end{question}

Note that an analogous phenomenon actually happens for linkless graphs: even though there are non-trivial links with vanishing linking number, an embedding of an intrinsically linked graph will always have a link with non-zero linking number.

Moreover, if the answer to \cref{q:van_Kampen} is affirmative, this would confirm \cref{conj:DY_YD} since both \DY- and \YDtrafos\ preserve the vanishing of the van Kampen obstruction.

\medskip
During our research we identified a number of 2-complexes whose~embeddability seems likewise out of reach of standard techniques, in particular, cannot be decided by the van Kampen obstruction.

Consider the complex $\mathcal J_{3,n}$ with 1-skeleton $K_3*K_n$ (\ie\ $K_3\cupdot K_{3,n}\cupdot K_n$) in which we attach a 2-cell  along $abcd$ for every four-tuple of  vertices $a,b,c,d$ with $a,b\in K_3$ and $c,d\in K_n$.
For $n\le 4$ the complex can be obtained from $\mathcal K_7-\Delta$ by cloning and joining 2-cells, hence can be embedded following \cref{ex:K_7-Delta}, \cref{res:joining_embeds} and \cref{res:cloning_embeds}.
This does not apply for $n\ge 5$.

\begin{question}
    \label{q:J3n}
    For which $n\ge 5$ does $\mathcal J_{3,n}$ embed in $\RR^4$?
\end{question}

Note that any two 2-cells of $\mathcal J_{3,n}$ intersect, hence the van Kampen obstruction is zero and gives no information about embeddability. 
The more interesting outcome would be that $\mathcal J_{3,n}$ does \emph{not} embed. 
Some consequences are:
\begin{myenumerate}
    \item this complex would replace the \FKT\ complex as the simplest known example of a non-embeddable complex with vanishing van Kampen obstruction.
    \item the \DYtrafo\ at $K_3\subset\mathcal J_{3,n}$ results in a complex in which every 2-cell passes through a fixed vertex. The \DYtrafo\ is therefore embeddable by \cref{res:2_dominating_vertices}. Thus $\mathcal J_{3,n}$ could replace and greatly simplify the construction of \cref{sec:YD_counterexample}.
\end{myenumerate}

If $\mathcal J_{3,n}$ turns out to be embeddable for all $n\ge 5$, a number of modifications are possible that are still of interest:
\begin{myenumerate}
    \item adding a 2-cell along $K_3\subset\mathcal J_{3,n}$. If this makes a difference (\ie\ $\mathcal J_{3,n}$~is~embeddable, but $\mathcal J_{3,n}+\Delta$ is not), then $\mathcal J_{3,n}$ would be a counterexample to \cref{conj:stellifying}.
    \item adding a cone over $K_5\subset\mathcal J_{3,5}$. Then, performing a \DYtrafo\ at $K_3\subset\mathcal J_{3,n}$ results in a complex in which every 2-cell passes through one of \emph{two} vertices. This still guarantees embeddability by \cref{res:2_dominating_vertices}.
\end{myenumerate}

For both modifications the van Kampen obstruction remains zero.


\medskip
Yet another family of relatively simple complexes with unknown embeddability are constructed as follows: let $\mathcal H_{2n-1}$ be the complex with 1-skeleton $K_{2n-1}$ and a 2-cell attached along each cycle of length $n$ (alternatively, each cycle of length $\ge n$).
Like before, any two cycles intersect and the van Kampen obstruction vanishes.
Also, once again, $\mathcal H_{2n-1}$ is embeddable for $n\le 4$ as it can be constructed from $\mathcal K_7-\Delta$ using our  embeddability preserving operations.
The following question remains:

\begin{question}
    For which $n\ge 5$ does $\mathcal H_{2n-1}$ embed in $\RR^4$?
\end{question}

\fi
\tempnewpage 

\par\bigskip
\noindent
\textbf{Acknowledgements.} 
We thank 
\Tam\ for her valuable input and the many discussions about this project.

\bibliographystyle{abbrv}
\bibliography{literature}

\addresseshere

\newpage
\appendix

\section{Computer code}
\label{sec:appenix_code}

We use the Mathematica package \texttt{YTYGraphTransforms.m} by Mike Pierce \cite{YTYGraphTransforms} to enumerate the graphs in the 
 \DY-family with given generators. 

\subsection{Enumerating Heawood graphs}
The list of Heawood graphs can be generated~as follows:
\begin{lstlisting}
Heawood = WyeTriangleWyeFamily[{
    CompleteGraph[7],
    CompleteGraph[{3, 3, 1, 1}]
}];
Length@Heawood (* output: 78 *)
\end{lstlisting}

\subsection{Verifying \cref{res:all_Heawood_are_excluded}}

The following code can be used to verify that each minor of a Heawood graph is 4-flat. 
The code iterates through all Heawood graphs $G$ and all~minors of the form $H\in\{G-e,G/e\}$. It then checks for each pair of vertices $v,w\in V(H)$ whether $H-\{v,w\}$ is planar. 
If one such pair is found, then $H$ is 4-flat by \cref{res:linkless_planar_outerplanar}.

\begin{lstlisting}
counterexampleFound = False;
Do[ (* for all Heawood graphs G *)
    minors = Join[
        DeleteDuplicates[
            Table[EdgeDelete[G,e], {e, EdgeList[G]}],
            IsomorphicGraphQ],
        DeleteDuplicates[
            Table[EdgeContract[G,e], {e, EdgeList[G]}],
            IsomorphicGraphQ]
    ];
    Do[ (* for all minors H of G *)
        vertexPairFound = False;
        Do[ (* for each pair of vertices in H *)
            If[PlanarGraphQ[VertexDelete[H, pair]],
                vertexPairFound = True;
                Break[];
            ];,
            {pair, Subsets[VertexList[H], {2}]}
        ]
        If[!vertexPairFound,
            counterexampleFound = True;
            Print["Counterexample found!"];  
                (* <-- this code is never reached *)
        ];,
        {H, minors}
    ];,
    {G, Heawood}
]
If[!counterexampleFound,
    Print["No counterexample found!"] (* <-- this code is reached *)
];
(* output: No counterexample found! *)
\end{lstlisting}

\subsection{The ``remaining Heawood graphs''}
\label{sec:appenix_remaining_Heawood}

The seven remaining Heawood graphs (in the sense of \cref{sec:Heawood_excluded_minor}, including $K_7$ and $K_{3,3,1,1}$) can be listed as follows:

\begin{lstlisting}
remaining = Select[Heawood, Min@VertexDegree[#] > 3 &];
Length@remaining (* output: 7 *)
\end{lstlisting}

\end{document}